\documentclass[reqno]{amsart}
\usepackage{amssymb}
\usepackage{graphicx}
\usepackage[usenames, dvipsnames]{color}
\usepackage{epstopdf}
\usepackage{verbatim}
\usepackage{mathrsfs}

\numberwithin{equation}{section}

\newtheorem{theorem}{Theorem}[section]
\newtheorem{corollary}[theorem]{Corollary}
\newtheorem{lemma}[theorem]{Lemma}
\newtheorem{proposition}[theorem]{Proposition}

\theoremstyle{definition}
\newtheorem{remark}[theorem]{Remark}

\theoremstyle{definition}

\theoremstyle{definition}

\makeatletter
\def\dashint{\operatorname%
{\,\,\text{\bf-}\kern-.98em\DOTSI\intop\ilimits@\!\!}}
\makeatother

\def\\det{\text{det}}

\def\ri{\mathrm{i}}

\def\.5{\frac{1}{2}}

\def\bR{\mathbb{R}}
\def\bZ{\mathbb{Z}}

\def\bN{\mathbb{N}}

\def\bC{\mathbb{C}}

\def\tG{\tilde{G}}

\def\fB{\mathfrak{B}}

\def\Re{\text{Re}\,}
\def\Im{\text{Im}\,}

\def\cD{\mathcal{D}}

\begin{document}

\title[On an elliptic equation arising from composite materials]
{On an elliptic equation arising from composite materials}

\author[H. Dong]{Hongjie Dong}
\address[H. Dong]{Division of Applied Mathematics, Brown University,
182 George Street, Providence, RI 02912, USA}
\email{Hongjie\_Dong@brown.edu}
\thanks{H. Dong was partially supported by the NSF under agreement DMS-1056737.}

\author[H. Zhang]{Hong Zhang}
\address[H. Zhang]{Division of Applied Mathematics, Brown University,
182 George Street, Providence, RI 02912, USA}
\email{Hong\_Zhang@brown.edu}
\thanks{H. Zhang was partially supported by the NSF under agreement DMS-1056737.}

\begin{abstract}
In this paper, we derive an interior Schauder estimate for the divergence form elliptic equation
\begin{equation*}
D_i(a(x)D_iu)=D_if_i
\end{equation*}
in $\mathbb{R}^2$, where $a(x)$ and  $f_i(x)$ are piecewise H\"older continuous in a domain containing two touching balls as subdomains. When $f_i\equiv 0$ and $a$ is piecewise constant, we prove that $u$ is piecewise smooth with bounded derivatives. This completely answers a question raised by Li and Vogelius \cite{LV00} in dimension 2.

\end{abstract}
\maketitle
\section{Introduction}
In this article, we consider second-order divergence type elliptic equations with discontinuous coefficients and data
\begin{equation}
L_{r_1,r_2}u:=D_i(a(x)D_i u)=D_if_i\quad \text{in}\,\, \mathcal{D},\label{eq 2.161}
\end{equation}
where $\mathcal{D}$ is a bounded subset of $\bR^2$,
$$
a(x)=a_0\chi_{B_{r_1}(0,r_1)\cup B_{r_2}(0,-r_2)}+\chi_{\bR^2 \setminus (B_{r_1}(0,r_1)\cup B_{r_2}(0,-r_2))},
$$
$a_0>0$ is a constant, $r_1, r_2\in (0,\infty)$, and $\chi$ is the indicator function. This problem was raised by  Bonnetier and Vogelius \cite{BonVog00}, and can be considered as a simplified model for composite media with closely spaced interfacial boundaries.  Here $\mathcal{D}$  models the cross-section of a fiber-reinforced composite and the balls $B_{r_1}(0,r_1)$ and $B_{r_2}(0,-r_2)$ represent the cross-sections of the fibers; the remaining subdomain represents the matrix surrounding the fibers. Moreover, $a(x)$ is the shear modulus, which is a constant on the fibers, and a different constant on the matrix surrounding the fibers. The function $u$ stands for the out of plane elastic displacement.

Elliptic equations and systems arising from elasticity have been studied by many authors. See, for instance, \cite{CKC86, BonVog00, LV00, Babu,PCV87,DongARMA12, XB13, MR3021549, MR3296149}. In \cite{CKC86}, Chipot, Kinderlehrer, and Vergara-Caffarelli considered  divergence type uniformly elliptic systems in a domain  $\mathcal{D}\subset\bR^d$ consisting of finite numbers of linearly elastic, homogeneous, parallel laminae, which models the equilibrium problem of  linear laminates. In \cite{LV00}, Li and Vogelius studied divergence type elliptic equations in a bounded domain $\mathcal{D}\subset \bR^d$, where $\mathcal{D}$ can be divided into finite numbers of subdomains with $C^{1,\alpha}$ boundaries. The coefficients of the equations and data are H\"older continuous in each subdomain up to the boundary, but may have jump discontinuities across the boundaries of the subdomains. Under these conditions, they proved a global $W^{1,\infty}$ estimate and a piecewise $C^{1,\beta}$ estimate, for any $\beta\le \frac{\alpha}{d(\alpha+1)}$. Notably, their estimate does not depend on the distance of discontinuous surfaces, which indicates that by an approximation argument, interfaces may touch each other, e.g., the geometry shown in Figure \ref{fig:1.1}.  Later, Li and Nirenberg \cite{LN03} extended the result in \cite{LV00} to elliptic systems under the same condition. They were able to improve the piecewise $C^{1,\beta}$ estimate in \cite{LV00} to any $\beta\in(0,\frac{\alpha}{2(1+\alpha)}]$.

Regarding the operator in \eqref{eq 2.161}, Bonnetier and Vogelius \cite{BonVog00} first considered the Dirichlet value boundary with $r_1=r_2=1$ and $f_i\equiv0$: $L_{1,1}u=0$ in $\mathcal{D}$ and $u=\phi$ on $\partial\mathcal{D}$. They showed a global regularity result that the solution $u \in W^{1,\infty}(\mathcal{D})$. Later, Li and Vogelius \cite{LV00} extended the result in \cite{BonVog00} and proved that when $r_1=r_2=1$,  $f_i\equiv 0$, and
$\mathcal{D}=B_{R_0}$ with $R_0$ sufficiently large, the weak solution $u$ is piecewise smooth, i.e.,
\begin{equation*}
u\in C^{\infty}(\overline{B_1(0,1)}),\,\, u\in C^{\infty}(\overline{B_1(0,-1)}),\,\, u\in C^{\infty}\big(K\setminus(B_1(0,1)\cup B_1(0,-1))\big),
\end{equation*}
where $K$ is any compact subset of $B_{R_0}$. Then they asked the following natural question: can we drop the condition that $R_0$ being sufficiently large?

 In our first result, we answer this question by proving that $R_0>2$ is sufficient to guarantee that $u$ is piecewise smooth in the interior of $B_{R_0}$.
 \begin{figure}
\begin{center}
\includegraphics[width=2in]{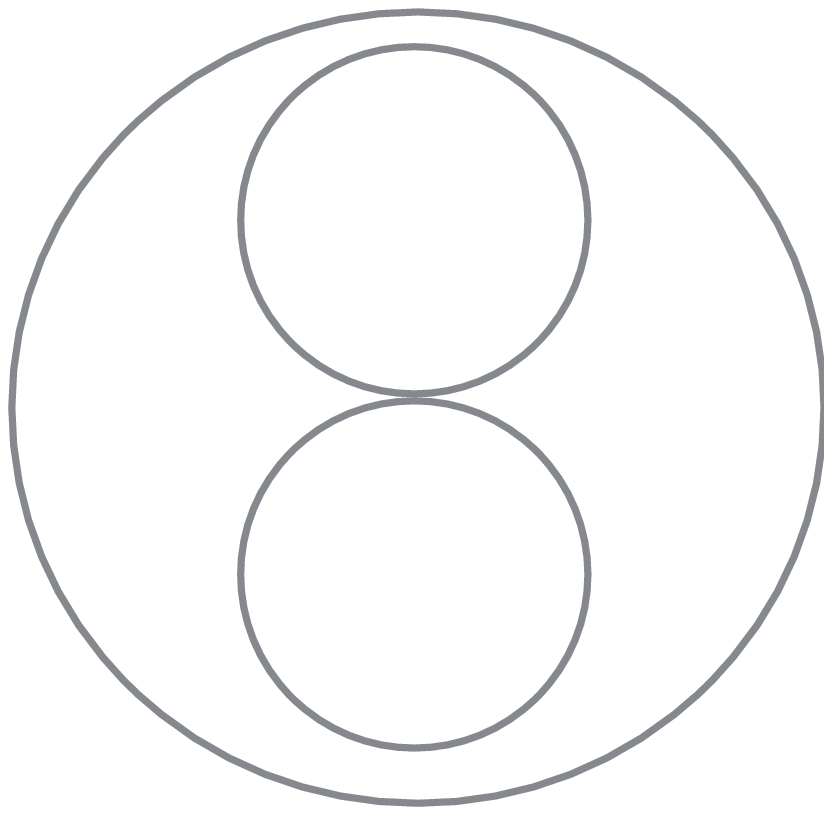}
\caption{}
\label{fig:1.1}
\end{center}
\end{figure}
\begin{theorem}\label{thm 2.161}
Let $R_0>2$ and $g\in H^{1/2}(\partial B_{R_0})$. Suppose $u$ is a weak solution of
\begin{equation*}
D_i(a(x)D_i u)=0\quad \text{in} \quad B_{R_0},\quad u=g\quad \text{on}\quad \partial B_{R_0},
\end{equation*}
where
\begin{align}
a(x)=a_0\chi_{B_1(0,1)\cup B_1(0,-1)}+\chi_{B_{R_0}\setminus(B_1(0,1)\cup B_1(0,-1))}.
\label{eq 2.171}
\end{align}
Then
\begin{equation*}
u\in C^{\infty}(\overline{B_1(0,1)}),\quad
u\in C^\infty(\overline{B_1(0,-1)}),\quad u\in C^{\infty}\big(K\setminus (B_1(0,1)\cup B_1(0,-1))\big)
\end{equation*}
for any compact set $K\subset B_{R_0}$.
\end{theorem}

To prove Theorem \ref{thm 2.161}, we borrow some ideas from \cite{LV00}. In \cite{LV00}, Li and Vogelius constructed a sequence of piecewise smooth solutions $\{u_j\}$ to
\begin{equation*}
D_i(a(x)D_iu)=0\quad \text{in}\,\, \bR^2,
\end{equation*}
the linear combinations of which are dense in $H_{\text{sym}}^s(\partial B_{R_0})$ for $R_0$ sufficiently large, where $H_{\text{sym}}^s(\partial B_{R_0})$ denotes the space of functions even in $x_1$ with finite $H^s$ norm for $s\ge 0$. The precise definition can be found in Section 2. Therefore, the solution $u$ to the Dirichlet problem with the boundary condition $u=\phi\in H_{\text{sym}}^s(\partial B_{R_0})$ can be approximated by linear combinations of $u_j$'s.  Hence, by a classical elliptic regularity argument,  one can show that $|D^ku|<\infty$  in each subdomain for any $k\ge 0$.

In this paper, we carry out a more careful analysis on $\{u_j\}$ to show that  $R_0>2$ is sufficient to guarantee that $\{u_j\}$  forms a Schauder basis for $H_{\text{sym}}^s(\partial B_{R_0})$. Precisely, it is obvious that
$$
\{e_j,j\ge 0\}:=\big\{(-1)^j\cos(2j\theta), (-1)^j\sin((2j+1)\theta),j\ge 0\big\}
$$
is an orthogonal basis of $H^s_{\text{sym}}(\partial B_{R_0})$.
Each $u_j$ can be written as a linear combination of $e_j$'s, i.e., $u_j=\sum_{k=0}^\infty M_{j,k}e_k$. We show that the infinite dimensional matrix $M:=(M_{k,j})_{k,j=0}^\infty$ define a bounded and invertible operator on a Hilbert space $\mathfrak{l}^s$. For the definition of $\mathfrak{l}^s$, see Section 2. An important observation in our proof is that the submatrix $\{M_{k,j}\}_{k,j=1}^\infty$ is diagonally dominant by column. From this, we deduce that the map induced by $M$ is invertible, which further implies that $\{u_j\}_{j\ge 0}$ forms a Schauder basis of $H^s_{\text{sym}}(\{|x|=R_0\})$.  The remaining proof then follows the lines in \cite{LV00}.

Another natural question to ask is if the geometry of the domain where the equation is satisfied affects the smoothness of the solution around the origin? In other words, if $Lu=0$ in $\mathcal{D}$, is it necessary that $\mathcal{D}$ contains a ball with radius $R_0>2$ for $u$ to be  piecewise smooth around the origin? Or $\mathcal{D}$ can be any neighborhood of the origin?
Our second result answers this question by proving an interior Schauder estimate for the non-homogeneous equation \eqref{eq 2.161} in a general domain. Furthermore, we break the symmetry of the coefficients, meaning that $a(x)$  can be two different positive constants  $a_0$ and $b_0$ in the two balls with different radii.
\begin{figure}
\begin{center}
\includegraphics[width=2in]{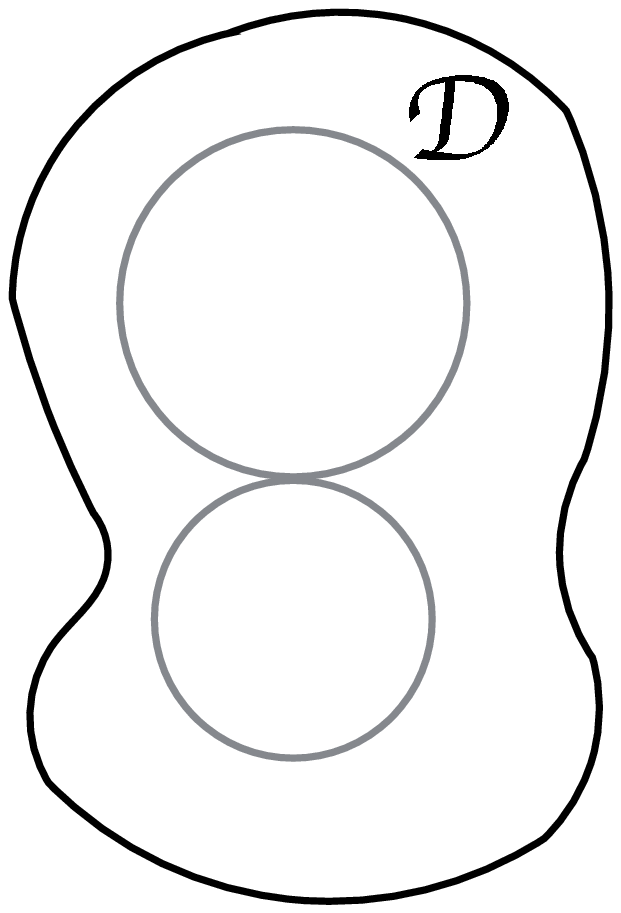}
\caption{}
\end{center}
\end{figure}
\begin{theorem}\label{thm 3.133}
Let $r_1,r_2\in (0,\infty)$, $\gamma\in (0,1)$, and $n\ge 0$ be an integer. Assume that $\mathcal{D}\subset \bR^2$ is a bounded open set. Suppose that for any $i$, $f_i$ is piecewise $C^{n,\gamma}$, i.e.,
$$
f_i\in C^{n,\gamma}(\cD\cap {B_{r_1}(0,r_1)}), \quad  f_i\in C^{n,\gamma}(\cD\cap {B_{r_2}(0,-r_2)}),
$$
and
$$
f_{i}\in C^{n,\gamma}({\cD\setminus(B_{r_1}(0,r_1)\cup B_{r_2}(0,-r_2))}).
$$
Let $u$ be a weak solution to
\begin{equation*}
D_i(a(x)D_i u(x))=D_if_i\quad \text{in}\,\, \mathcal{D},
\end{equation*}
where
$$
a(x)=a_0\chi_{B_{r_1}(0,r_1)}+b_0\chi_{B_{r_2}(0,-r_2)}+\chi_{\bR^2\setminus (B_{r_1}(0,r_1)\cup B_{r_2}(0,-r_2) )}
$$
Define $\mathcal{D}_{\varepsilon}=\{x\in \mathcal{D},\text{dist}(x,\partial\mathcal{D})\ge \varepsilon\}$ for any $\varepsilon>0$.
Then
$$
u\in C^{n+1,\gamma}(\mathcal{D}_{\varepsilon}\cap B_{r_1}(0,r_1)),\quad u\in C^{n+1,\gamma}(\mathcal{D}_{\varepsilon}\cap B_{r_2}(0,-r_2)),
$$
and
$$
u\in C^{n+1,\gamma} (\mathcal{D}_{\varepsilon}\setminus(B_{r_1}(0,r_1)\cup B_{r_2}(0,-r_2))).
$$
In particular, when $f_i\equiv 0$, $u$ is piecewise smooth in $\mathcal{D}_{\varepsilon}$ up to the boundary.
\end{theorem}

For the proof, first we find a conformal mapping which maps two balls with different radii to two balls with the same radius, so it is sufficient to consider $r_1=r_2$ and we denote the elliptic operator in \eqref{eq 2.161} with $r_1=r_2=1$ by $L$. Then the conformal mapping $\Gamma: z\rightarrow \ri/z$ maps $\{|z-\ri y|=1\}$ and  $\{|z+\ri y|=1\}$ to $\{\Re z=\.5\}$ and $\{\Re z=-\.5\}$, respectively, where $\ri=\sqrt{-1}$ is the imaginary unit. We are able to construct Green's function $\tilde{G}(x,y)$ of the elliptic operator $\tilde Lu=D_i(A(x)D_i u)$, where
$$
A(x)=a_0\chi_{\{x_1>\.5\}}+b_0\chi_{\{x_1<-\.5\}}+\chi_{\{|x_1|<\.5\}}.
$$
With the help of $\Gamma$ and $\tilde{G}(x,y)$, we obtain Green's function $G (x,y)$  of the elliptic operator $L$ in $\bR^2$, which can be written as an infinite series of logarithmic function composed with smooth functions, for example,  when $x\in \bR^2\setminus (B_1(0,1)\cup B_1(0,-1))$, and $y\in B_1(0,1)$,
$$
G (x,y)=c_1\sum_{k=0}^\infty (\alpha\beta)^{k}\log|X_{-2k}(x)-y|-c_2\sum_{k=1}^\infty (\alpha\beta)^{k-1}\log|X_{2k-1}(x)-\overline{y}|,
$$
where $c_1, c_2, \alpha,$ and $\beta$ are constants with $|\alpha|,|\beta|<1$, $\overline{y}=(y_1,-y_2)$, and $\{X_k\}$ are conformal maps and $X_0(x)=x$. Note that $\log|x-y|$ is Green's function of the Laplacian up to  a factor. This observation allows us to implement some known results of the Laplace equation with piecewise H\"older continuous data on the right-hand side.  More precisely, the original problem is decomposed to understand the regularity of solutions to the following equations
\begin{align*}
&\Delta u=D_i(f_i\chi_{B_1(0,1)}),\\
&\Delta u=D_i(f_i\chi_{B_1(0,-1)}),\\
&\Delta u=D_i(f_i\chi_{\bR^2\setminus(B_1(0,1)\cup B_1(0,-1))}),
\end{align*}
where in each subdomain $f_i \in C^{n,\gamma}$. By locally flattening the boundary, the first two equations can be further reduced to the case that $f_i\in C^{n,\gamma}$ in two half spaces, i.e., $\{x_2>0\}$ and $\{x_2<0\}$. The detail can be found in the proof of Theorem \ref{thm 2.191} Case 1. The last equation needs an extension result to be reduced to the previous case. See Lemma \ref{lemma 3.231}.
Combining with the smoothness of each $X_k$, we are able to estimate all the derivatives of the solution.


By a standard perturbation argument, we have the following corollary.
\begin{corollary}\label{cor 3.71}
Let $0<\lambda\le \Lambda<\infty$ and $n\ge 0$ be an integer.
Assume $a(x)$ is piecewise $C^{n,\gamma}$, i.e., $$a(x)\in C^{n,\gamma}(B_{r_1}(0,r_1)),\quad a(x)\in C^{n,\gamma}(B_{r_2}(0,-r_2)),$$
and
$$
a(x)\in C^{n,\gamma}(\bR^2\setminus(B_{r_1}(0,r_1)\cup B_{r_2}(0,-r_2))),
$$
and satisfies the ellipticity condition $\lambda\le a\le \Lambda$.
Suppose that for each $i$, $f_i$ is piecewise $C^{n,\gamma}$, i.e.,
$$
f_i\in C^{n,\gamma}(\cD\cap B_{r_1}(0,r_1)),\quad f_i\in C^{n,\gamma}(\cD\cap B_{r_2}(0,-r_2)),
$$
and
$$
f_i\in C^{n,\gamma}(\cD\setminus(B_{r_1}(0,r_1)\cup B_{r_2}(0,-r_2))).
$$
Let $u$ be a weak solution to
\begin{equation*}
D_i(a(x)D_i u)=D_if_i
\end{equation*}
in $\mathcal{D}\subset \bR^2$. Denote $\mathcal{D}_\varepsilon=\{x\in \mathcal{D},\text{dist}(x,\partial\mathcal{D})\ge \varepsilon\}$ for $\varepsilon>0$.
Then $u$ is piecewise $C^{n+1,\gamma}$ in $\mathcal{D}_\varepsilon$ up to the boundary.
\end{corollary}

This paper is organized as follows. In the next section, we introduce some notation and preliminary results, which are needed in the proof of our main theorems. In Section 3, we prove Theorem \ref{thm 2.161}. In Section 4, we make necessary preparations and prove Theorem \ref{thm 3.133} and Corollary \ref{cor 3.71}.

\section{Notation and preliminary results}
In this section, we first introduce some notation used throughout this paper. The Einstein summation convention is applied in this paper.
We use $B_{R}(x)$ to denote the Euclidean ball in $\bR^2$ with radius $R$ and center $x$. For simplicity, $B_1(0,1)$ and $B_1(0,-1)$ are denoted by $\fB_1$ and $\fB_2$, respectively, and $\bR^2\setminus\overline{\fB_1\cup \fB_2}=\fB_0$.  When there is no confusion, we use $B_{R_0}$ to denote the ball with radius $R_0$ and center $(0,0)$.  We use $L$ to denote the operator $L_{r_1,r_2}$ when $r_1=r_2=1$.

Let $\mathcal{D}$ be a subset of $\bR^2$ and $\beta\in (0,1]$. For any function $f$, we define
\begin{equation*}
[f]_{\beta;\mathcal{D}}=\sup\left\{\frac{|f(x)-f(y)|}{|x-y|^\beta}: x,y\in \mathcal{D}, x\neq y\right\}
\end{equation*}
and
$$\|f\|_{\beta;\mathcal{D}}=\|f\|_{L_\infty;\mathcal{D}}+[f]_{\beta;\mathcal{D}}.$$
We denote the space corresponding to $\|\cdot\|_{\beta;\mathcal{D}}$ by $C^{\beta}(\mathcal{D})$. For nonnegative integer $m$, we define
$$
\|f\|_{m,\beta;\mathcal{D}}
=\|f\|_{L_\infty;\mathcal{D}}+[D^mf]_{\beta;\mathcal{D}}.
$$
The space corresponding to $\|\cdot\|_{m,\beta;\mathcal{D}}$ is denoted by $C^{m,\beta}(\mathcal{D})$.

Denote $L_{\text{sym}}^2\big(\{x:|x|=R_0\}\big)$ to be the set of real-valued $L^2$ functions on the circle $\{|x|=R_0\}$ which are even with respect to $x_1$. We use a similar notation for the Sobolev spaces $H^s_{\text{sym}}\big(\{x:|x|=R_0\}\big)$ for $s \ge 0$. Note that $L_{\text{sym}}^2\big(\{x:|x|=R_0\}\big)=H^0_{\text{sym}}\big(\{x:|x|=R_0\}\big)$.

Let $V$ be a Banach space over $\bR$. We say that a sequence  $\{b_n\}$  in $V$ is a Schauder basis of $V$ if for every  $v\in V$, there exists a unique sequence $\{a_n\}$ of scalars such that
$$v=\sum_{n=0}^\infty a_nb_n,$$
where the convergence is in the norm topology.

We first prove an extension lemma, which is useful in our proofs.
\begin{lemma}\label{lemma 3.231}
Let $\gamma\in (0,1)$, $n\ge 0$, and $f\in C^{n,\gamma}(\bR^2\setminus(B_{r_1}(0,r_1)\cup B_{r_2}(0,-r_2)))$. Then there exists  a function $F\in C^{n,\gamma}(\bR^2)$ such that $$F|_{\bR^2\setminus(B_{r_1}(0,r_1)\cup B_{r_2}(0,-r_2))}=f$$ and $$\|F\|_{n,\gamma;\bR^2}\le C\|f\|_{n,\gamma;\bR^2\setminus(B_{r_1}(0,r_1)\cup B_{r_2}(0,-r_2))},$$
where $C$ is  independent of $f$.
\end{lemma}
\begin{proof}
It suffices to consider the extension in $B_{r_1}(0,r_1)$.  From  \cite[Theorem 2.19]{GGS91} and \cite[Theorem 9.3]{ADN}, for any $n\ge 0$ there exists a solution $\tilde{f}$ to the equation
\begin{align*}
&(-\Delta)^{n+1}\tilde{f}=0\quad \text{in}\,\, B_{r_1}(0,r_1),\\
&\tilde{f}=f,\,\, D_\nu^1\tilde{f}=D_\nu^1f,\ldots,\,D_\nu^{n}\tilde{f}=D_\nu^{n}f\quad \text{on}\quad \partial B_{r_1}(0,r_1),
\end{align*}
where $\nu$ is the unit normal vector of $\partial B_{r_1}(0,r_1)$. Moreover,
$$\|\tilde{f}\|_{n,\gamma;B_{r_1}(0,r_1)}\le C\sum_{i=0}^{n}\|D_\nu^i f\|_{n-i,\gamma;\partial B_{r_1}(0,r_1)}\le C\|f\|_{n,\gamma;\bR^2\setminus(B_{r_1}(0,r_1)
\cup B_{r_2}(0,-r_2))},$$
where $C$ is independent of $f$.
Similarly, the extension of $f$ to $B_{r_2}(0,-r_2)$ is denoted by $\hat{f}$.  Finally we define
\begin{align*}
F=\tilde{f}\chi_{B_{r_1}(0,r_1)}+\hat{f}\chi_{B_{r_2}(0,-r_2)}+f\chi_{\bR^2\setminus(B_{r_1}(0,r_1)\cup B_{r_2}(0,-r_2))}.
\end{align*}
It is easy to see that $F$ is the desired function.
\end{proof}

Denote
$$
\alpha=\frac{a_0-1}{a_0+1}\in (-1,1).
$$
Let $a(x)$ be defined as in \eqref{eq 2.171}.
As mentioned in the introduction, Li and Vogelius \cite{LV00} constructed  a sequence solutions to
\begin{equation}
D_i(a(x)D_i u)=0\label{eq 1.221}
\end{equation}
in $\bR^2$, whose linear span is dense in $H^s_{\text{sym}}(\{x\,:\,|x|=R_0\})$ for sufficiently large $R_0$.
Following \cite{LV00}, we define $\Psi_j$ as follows:
\begin{align*}
\Psi_j(z)&=\frac {2} {a_0+1}\sum_{k=0}^\infty \alpha^k \frac {z^j} {(kz+\ri)^j}\quad\text{in}\,\, \{z: |z-\ri|<1\},\\
\Psi_j(z)&=(-1)^{\frac{j+1}{2}}\ri z^j+\sum_{k=1}^\infty\alpha^k
\left[\frac{z^j}{(kz+\ri)^j}-\frac {z^j} {(kz-\ri)^j}\right]\\
	    &\text{in} \quad\{z: |z+\ri|>1\, \text{and}\, |z-\ri|>1\},\\
\Psi_j(z)&=-\frac{2}{a_0+1}\sum_{k=0}^\infty \alpha^k\frac{z^j}{(kz-\ri)^j}\quad \text{in}\,\, \{z:|z+\ri|<1\}
\end{align*}
\underline{for $j$ odd}, and
\begin{align*}
\Psi_j(z)& =\frac{2}{a_0+1}\sum_{k=0}^\infty (-\alpha)^k\frac{z^j}{(kz+\ri)^j}\quad \text{in}\,\,\{z:|z-\ri|<1\},\\
\Psi_j(z)&=(-1)^{j/2}z^j+\sum_{k=1}^\infty (-\alpha)^k\left[\frac{z^j}{(kz+\ri)^j}+\frac{z^j}{(kz-\ri)^j}\right]\\
&\text{in}\,\, \{z: |z+\ri|>1\,\text{and}\,|z-\ri|>1\},\\
\Psi_j(z)&=\frac{2}{a_0+1}\sum_{k=0}^\infty (-\alpha)^k \frac{z^j}{(kz-\ri)^j}\quad \text{in} \quad\{z: |z+\ri|<1
\end{align*}
\underline{for $j$ even}. Let $u_j(x_1,x_2)=R_0^{-j}\Re\Psi_j(z)$. It is shown in \cite[Proposition 8.2]{LV00} that $\{u_j\}$ are solutions to \eqref{eq 1.221}. In the lemma below, we first give an explicit representation of $u_j$'s on $\{|x|=R_0\}$ in terms of trigonometric polynomials for $R_0>2$.

\begin{lemma}
                        \label{lem2.2}
Let $R_0>2$ be a constant. For any $j\ge 0$, we have
\begin{align}\nonumber
&u_{2j+1}(x_1,x_2)|_{|x|=R_0}\\
&=(-1)^j\sin(2j+1)\theta-2\sum_{k=1}^\infty\sum_{l=0}^\infty \frac{\alpha^k{{2l+2j+1}\choose {2j}}(-1)^l\sin(2l+1)\theta}{(kR_0)^{2j+2l+2}};\label{eq 3.1311}
\end{align}
for any $j\ge 1$, we have
\begin{align}\label{eq 3.1312}
&u_{2j}(x_1,x_2)|_{|x|=R_0}\nonumber\\
&=(-1)^j \cos2j\theta+2\sum_{k=1}^\infty\sum_{l=0}^\infty \frac{(-\alpha)^k{{2l+2j-1}\choose {2j-1}}(-1)^l\cos (2l\theta)}{(kR_0)^{2l+2j}},
\end{align}
and
\begin{equation}
                \label{eq1.20}
u_0=\frac{1-\alpha}{1+\alpha}=\frac{1}{a_0}.
\end{equation}
\end{lemma}
\begin{proof}
From the definition, for any $j\ge 0$ and $x\in \fB_0$, we have
\begin{align*}
&u_{2j+1}(x_1,x_2)=R_0^{-(2j+1)}\Re\Psi_{2j+1}(z)\\
&=R_0^{-(2j+1)}\Re \left\{(-1)^{j+1}\ri z^{2j+1}+\sum_{k=1}^\infty\alpha^k
\Big[\frac{z^{2j+1}}{(kz+\ri)^{2j+1}}-\frac{z^{2j+1}}{(kz-\ri)^{2j+1}}\Big]\right\},\\
&u_{2j}(x_1,x_2)=R_0^{-2j}\Re \Psi_{2j}(z)\\
&=R_0^{-2j}\Re
\left\{(-1)^{j}z^{2j}+\sum_{k=1}^\infty(-\alpha)^k
\Big[\frac{z^{2j}}{(kz+\ri)^{2j}}+\frac{z^{2j}}{(kz-\ri)^{2j}}\Big] \right\}.
\end{align*}
Set $z=R_0e^{\ri\theta}$ and we have
\begin{align}
\nonumber
&u_{2j+1}(x_1,x_2)|_{|x|=R_0}\nonumber\\
&=\Re\left\{(-1)^{j+1}\ri e^{\ri(2j+1)\theta}
+\sum_{k=1}^\infty\alpha^k\Big[\frac{e^{\ri(2j+1)\theta}}
{(kR_0e^{\ri\theta}+\ri)^{2j+1}}-\frac{e^{\ri(2j+1)\theta}}
{(kR_0e^{\ri\theta}-\ri)^{2j+1}}\Big]\right\}\nonumber\\
&=(-1)^j\sin(2j+1)\theta+\Re \sum_{k=1}^\infty \alpha^k(kR_0)^{-(2j+1)}\Big[\Big(\frac{1}{1+A_k}\Big)^{2j+1}-\Big(\frac{1}{1-A_k}\Big)^{2j+1}\Big],\label{eq 1.051}
\end{align}
where $A_k=\ri/{(kR_0e^{\ri\theta})}$. It is obvious that $|A_k|<1$ for any $k\ge 1$, and
\begin{equation*}
\Big(\frac{1}{1+A_k}\Big)^{2j+1}=\sum_{l=0}^\infty {{l+2j}\choose {2j}}(-A_k)^l,\quad
\Big(\frac{1}{1-A_k}\Big)^{2j+1}=\sum_{l=0}^\infty {{l+2j}\choose {2j}}A_k^l,
\end{equation*}
where ${m\choose n}=\frac{m!}{n!(m-n)!}$ for $m\ge n$. Therefore,
\begin{equation*}
\Big(\frac{1}{1+A_k}\Big)^{2j+1}-\Big(\frac{1}{1-A_k}\Big)^{2j+1}=-2\sum_{l=0}^\infty {{2l+1+2j}\choose {2j}} A_k^{2l+1}.
\end{equation*}
Plugging the formula above into \eqref{eq 1.051}, we get \eqref{eq 3.1311} for $j\ge 0$. In the same way, for $j\ge 1$, we  obtain \eqref{eq 3.1312}. Finally, \eqref{eq1.20} follows from a simple calculation. The lemma is proved.
\end{proof}

Set
$$
e_{2j-1}=(-1)^{j-1}\sin(2j-1)\theta,\quad e_{2j}=(-1)^j\cos2j\theta
$$
for $j\ge 1$, and $e_0=u_0$. Clearly, for any $s\ge 0$, $\{e_j\}_{j=0}^\infty$
forms an orthogonal basis of $H^s_{\text{sym}}(\{x:|x|=R_0\})$. It is easily seen that for any $s\ge0$, $u_j\in H^s_{\text{sym}}\big(\{|x|=R_0\}\big)$ and $$
\|u_j\|_{H^s(\{|x|=R_0\})}\le C_s (j+s)^s.
$$
Define the Hilbert space
$$
\mathfrak{l}^s=\Big\{a=(a_0,a_1,\ldots):\sum_{j=0}^\infty|a_j|^2
(1+j)^{2s}<\infty\Big\}
$$
with the norm $\|a\|_{\mathfrak{l}^s}=\big(\sum_{j=0}^\infty|a_j|^2(1+j)^{2s}\big)^{1/2}$. Then up to a constant factor, $H_{\text{sym}}^s(\{|x|=R_0\})$ is isometric to $\mathfrak{l}^s$.
Denote $\vec{f}\in\mathfrak{l}^s$ to be the infinite column vector $(f_j)_{j\ge 0}$. Let $M$ be an infinite dimensional matrix such that its $j$th column is $\vec{u}_j$.

Since $u_0=e_0$, we have $M_{0,0}=1$ and $M_{j,0}=0$ for $j\ge 1$.
By Lemma \ref{lem2.2}, $M=id+B$, where $id$ is the identity matrix, and $B$ is defined as follows: for $l,j\ge 1$
\begin{align}
                \label{eq 1.074}
B_{2l-1,2j-1}=-2\sum_{k=1}^\infty\frac{\alpha^k{{2l+2j-3}\choose {2j-2}}}{(kR_0)^{2l+2j-2}},\quad B_{2l,2j}=2\sum_{k=1}^\infty\frac{(-\alpha)^k{{2l+2j-1}\choose {2j-1}}}{(kR_0)^{2l+2j}},
\end{align}
$$
B_{0,2j}=2\sum_{k=1}^\infty\frac{(-\alpha)^k}{(kR_0)^{2j}},
\quad B_{2l-1,2j}=B_{2l,2j-1}=0,\quad B_{l,0}=0.
$$

The following observation is crucial in our proof.
\begin{lemma}
                                    \label{lem2.3}
When $R_0>2$, the infinite dimensional matrix $\{M_{i,j}\}_{i,j=1}^\infty$ is diagonally dominant by column.
\end{lemma}
\begin{proof}
Since $M=id+B$, it suffices to show that
$\sum_{l=1}^\infty|B_{l,j}|<1$ for $j\ge 1$.
We first consider odd number columns.  When $\alpha\in (0,1)$, obviously  $B_{2l-1,2j-1}<0$. On the other hand, when $\alpha\in (-1,0)$, by \eqref{eq 1.074} we have $B_{2l-1,2j-1}>0$.
Thus, for $j\ge 1$ we have
\begin{align}
                    \label{eq 3.121}
\sum_{l=1}^\infty|B_{2l-1,2j-1}|=\sum_{l=1}^\infty2\Big|\sum_{k=1}^\infty\frac{\alpha^k{{2l+2j-3}\choose {2j-2}}}{(R_0k)^{2l+2j-2}}\Big|=2\Big|\sum_{l=1}^\infty\sum_{k=1}^\infty \frac{\alpha^k{{2l+2j-3}\choose {2j-2}}}{(R_0k)^{2l+2j-2}}\Big|.
\end{align}
Note that
\begin{align}\nonumber
&2\sum_{l=1}^\infty\sum_{k=1}^\infty \frac{\alpha^k{{2l+2j-3}\choose {2j-2}}}{(R_0k)^{2l+2j-2}}=2\sum_{k=1}^\infty\frac{\alpha^k}{(kR_0)^{2j-1}}\sum_{l=1}^\infty \frac{{{2l+2j-3}\choose {2j-2}}}{(R_0k)^{2l-1}}\\\nonumber
&=\sum_{k=1}^\infty\frac{\alpha^k}{(kR_0)^{2j-1}}\sum_{l=0}^{\infty}{{2j-2+l}\choose {2j-2}}\left(\Big(\frac{1}{kR_0}\Big)^l-\Big(-\frac{1}{kR_0}\Big)^l\right)\\\nonumber
&=\sum_{k=1}^\infty \frac{\alpha^k}{(kR_0)^{2j-1}}
\left(\Big(\frac{1}{1-1/(kR_0)}\Big)^{2j-1}-\Big(\frac{1}{1+1/(kR_0)}\Big)^{2j-1}\right)\\
&=\sum_{k=1}^\infty\alpha^k
\left(\Big(\frac{1}{kR_0-1}\Big)^{2j-1}-\Big(\frac{1}{kR_0+1}\Big)^{2j-1}\right)\label{eq 4.32}.
\end{align}
Since $|\alpha|<1$ and $R_0>2$, we have
\begin{align}\nonumber
&\left|\sum_{k=1}^\infty\alpha^k\left(\Big(\frac{1}{kR_0-1}\Big)^{2j-1}
-\Big(\frac{1}{kR_0+1}\Big)^{2j-1}\right)\right|\\
&<\sum_{k=1}^\infty\left(\Big(\frac{1}{2k-1}\Big)^{2j-1}
-\Big(\frac{1}{2k+1}\Big)^{2j-1}\right)=1.\label{eq 3.123}
\end{align}
Then combining \eqref{eq 3.121}-\eqref{eq 3.123}, we obtain
\begin{equation*}
\sum_{l=1}^\infty|B_{2l-1,2j-1}|<1.
\end{equation*}
Similarly, for $j\ge 1$ we compute
\begin{align}\label{eq 4.31}
&\sum_{l=1}^\infty|B_{2l,2j}|\le \sum_{l=1}^\infty\sum_{k=1}^\infty 2\frac{{|\alpha|}^k {{2l+2j-1}\choose {2j-1}}}{(kR_0)^{2l+2j}}\\\nonumber
&=\sum_{k=1}^\infty\frac{|\alpha|^k}{(kR_0)^{2j}}\sum_{l=1}^\infty 2\frac{{{2l+2j-1}\choose {2j-1}}}{(kR_0)^{2l}}\\\nonumber
&=\sum_{k=1}^\infty\frac{|\alpha|^k}{(kR_0)^{2j}}2
\left(\sum_{l=0}^\infty\frac{{{2l+2j-1}\choose {2j-1}}}{(kR_0)^{2l}}-1\right)\\\nonumber
&=\sum_{k=1}^\infty\frac{|\alpha|^k}{(kR_0)^{2j}}
\left(\Big(\frac{1}{1-1/(kR_0)}\Big)^{2j}+\Big(\frac{1}{1+1/(kR_0)}\Big)^{2j}-2\right)\\\label{eq 1.271}
&=\sum_{k=1}^\infty|\alpha|^k \left(\Big(\frac{1}{kR_0-1}\Big)^{2j}+\Big(\frac{1}{kR_0+1}\Big)^{2j}
-2\Big(\frac{1}{kR_0}\Big)^{2j}\right).
\end{align}
Note that for any $k\ge 1$, $R_0>2$, and $j\ge 1$, by convexity,
$$
\Big(\frac{1}{kR_0-1}\Big)^{2j}+\Big(\frac{1}{kR_0+1}\Big)^{2j}
-2\Big(\frac{1}{kR_0}\Big)^{2j}>0.
$$
Therefore,  the right-hand side of \eqref{eq 1.271} is less than
\begin{equation*}
\sum_{k=1}^\infty\Big(\frac{1}{kR_0-1}\Big)^{2j}+\Big(\frac{1}{kR_0+1}\Big)^{2j}
-2\Big(\frac{1}{kR_0}\Big)^{2j},
\end{equation*}
which is decreasing with respect to $R_0$ because the right-hand side of \eqref{eq 4.31} is decreasing. Thus,  the left-hand side of \eqref{eq 1.271} with $|\alpha|=1$ is less than
\begin{align*}
&\sum_{k=1}^\infty \Big(\frac{1}{2k-1}\Big)^{2j}+\Big(\frac{1}{2k+1}\Big)^{2j}
-2\Big(\frac{1}{2k}\Big)^{2j}\\
&=1+2\sum_{k=1}^\infty\left(\Big(\frac{1}{2k+1}\Big)^{2j}-\Big(\frac{1}{2k}\Big)^{2j}\right)<1.
\end{align*}
Therefore,
\begin{equation*}
\sum_{l=1}^\infty|B_{2l,2j}|<1.
\end{equation*}
The lemma is proved.
\end{proof}

In \cite[Proposition 8.5]{LV00}, it is proved that for $R_0$ sufficiently large depending on $s$, $\text{span}\{u_j|_{\{x: |x|=R_0\}}\}$ is dense in $H_{\text{sym}}^s(\{|x|=R_0\})$. In the following proposition, by using Lemma \ref{lem2.3} we prove that $R_0>2$ is sufficient to show that $\{u_j|_{\{x: |x|=R_0\}}\}$ forms a Schauder basis in $H_{\text{sym}}^s(\{|x|=R_0\})$ for any $s\ge 0$.

\begin{proposition}\label{prop 1.291}
For any $R_0>2$ and $s\ge 0$, $\{u_j|_{\{x: |x|=R_0\}}\}$ forms a Schauder basis in $H_{\text{sym}}^s(\{|x|=R_0\})$, i.e., for any $f\in H_{\text{sym}}^s(\{|x|=R_0\})$, there exists a unique sequence $\{a_j\}_{j=0}^\infty$ in $\bR$ such that $f=\sum_{j=0}^\infty a_ju_j$. Moreover, we have
\begin{equation}
                                \label{eq12.44}
\sum_{j=0}^\infty a_j^2 (1+j)^{2s}\le C(s,R_0)\|f\|_{H^s(\{|x|=R_0\})}^2.
\end{equation}
\end{proposition}
Before proving Proposition \ref{prop 1.291}, we first show that the matrix $M$  defines a bounded and invertible operator on $\mathfrak{l}^s$ in the following lemmas.


\begin{lemma}\label{lem 2.5}
The matrix $M$ defines a bounded operator on $\mathfrak{l}^s$.
\end{lemma}

\begin{proof}
For any nonnegative integer $N$, we first estimate the block $(B_{l,j})_{l\ge N,j\ge 0}$. Clearly,
\begin{align}
\sum_{l\ge N}\sum_{j=0}^\infty|B_{l,j}|= \sum_{2l\ge N}\sum_{j=0}^\infty|B_{2l,2j}|+\sum_{2l-1\ge N}\sum_{j=1}^\infty|B_{2l-1,2j-1}|.\label{eq 3.124}
\end{align}
From \eqref{eq 1.074} and the fact that $B_{l,0}=0$ for $l\ge 0$, we have
\begin{align}\nonumber
 &\sum_{2l\ge N}\sum_{j=0}^\infty|B_{2l,2j}|= \sum_{2l\ge N}\sum_{j=1}^\infty2\left|\sum_{k=1}^\infty\frac{(-\alpha)^k{{2l+2j-1}\choose {2l}}}{(kR_0)^{2l+2j}}\right|\\
 \nonumber
& \le \sum_{2l\ge N}\sum_{j=1}^\infty2\sum_{k=1}^\infty\frac{|\alpha|^k{{2l+2j-1}\choose {2l}}}{(kR_0)^{2l+2j}}\\
\nonumber
 &\le \sum_{2l\ge N}\sum_{k=1}^{\infty}|\alpha|^k
\left(\Big(\frac{1}{kR_0-1}\Big)^{2l+1}-\Big(\frac{1}{kR_0+1}\Big)^{2l+1}\right)\\
 \nonumber
 &=\sum_{2l\ge N}\frac{1}{R_0^{2l+1}}\sum_{k=1}^\infty|\alpha|^k
\left(\Big(\frac{1}{k-1/R_0}\Big)^{2l+1}-\Big(\frac{1}{k+1/R_0}\Big)^{2l+1}\right).
\end{align}
Since $R_0>2$,
\begin{align*}
&\sum_{k=1}^\infty|\alpha|^k
\left(\Big(\frac{1}{k-1/R_0}\Big)^{2l+1}-\Big(\frac{1}{k+1/R_0}\Big)^{2l+1}\right)\\
&\le2 \sum_{k=1}^\infty|\alpha|^k\Big(\frac{1}{k-1/2}\Big)^{2l+1}\le C(\alpha),
\end{align*}
where $C(\alpha)$ only depends on $\alpha$. Then we get
\begin{equation}
\sum_{2l\ge N}\sum_{j=0}^\infty|B_{2l,2j}|\le
\sum_{2l\ge N}\frac{C(\alpha)}{R_0^{2l+1}}
\le C(\alpha)R_0^{-N-1}.\label{eq 1.072}
\end{equation}
Similarly by \eqref{eq 1.074}, we obtain
\begin{align}
\nonumber
&\sum_{2l-1\ge N}\sum_{j=1}^\infty|B_{2l-1,2j-1}|\\
&\le\sum_{2l-1\ge N}\sum_{k=1}^\infty|\alpha|^k
\left(\Big(\frac{1}{kR_0-1}\Big)^{2l}+\Big(\frac{1}{kR_0+1}\Big)^{2l}\right)\le CR_0^{-N-1}.\label{eq 1.071}
\end{align}
Combining \eqref{eq 3.124}, \eqref{eq 1.072}, and \eqref{eq 1.071}, we have
\begin{equation}
\sum_{l\ge N}\sum_{j=0}^{\infty}|B_{l,j}|\le CR_0^{-N-1}.\label{eq 1.153}
\end{equation}

Now we are ready to show that $M$ is a bounded operator on $\mathfrak{l}^s$.  For any $f\in \mathfrak{l}^s$ with
\begin{equation*}
\|f\|_{\mathfrak{l}^s}^2=\sum_{l=0}^\infty|f_l|^2(1+l)^{2s}=1,
\end{equation*}
we have
\begin{align}
\|Mf\|_{\mathfrak{l}^s}^2=\sum_{l=0}^\infty(1+l)^{2s}|(Mf)_l|^2=\sum_{l=0}^\infty(1+l)^{2s}
\Big|\sum_{j=0}^\infty M_{l,j}f_j\Big|^2.\label{eq 3.129}
\end{align}
Since $M_{l,j}=\delta_{lj}+B_{l,j}$, by the Cauchy-Schwarz inequality
\begin{equation}\label{eq 3.128}
\Big|\sum_{j=0}^\infty M_{l,j}f_j\Big|^2
=\Big|f_l+\sum_{j=0}^\infty B_{l,j}f_j\Big|^2\le 2\Big(f_l^2+\Big|\sum_{j=0}^\infty B_{l,j}f_j\Big|^2\Big).
\end{equation}
From \eqref{eq 3.129} and \eqref{eq 3.128}, we get
$$
\|Mf\|_{\mathfrak{l}^s}^2\le 2+2\sum_{l=0}^\infty(1+l)^{2s}\Big|\sum_{j=0}^\infty B_{l,j}f_j\Big|^2.
$$
Notice that \eqref{eq 1.153} implies for $l\ge 0$,
\begin{equation}\label{eq 3.1210}
\sum_{j=0}^\infty|B_{l,j}|\le CR_0^{-l-1}.
\end{equation}
Combining with fact that $\|f\|_{l_\infty}\le 1$, we obtain
\begin{align*}
&\sum_{l=0}^\infty(1+l)^{2s}\Big|\sum_{j=0}^\infty B_{l,j}f_j\Big|^2
\le \sum_{l=0}^\infty (1+l)^{2s}\Big(\sum_{j=0}^\infty |B_{l,j}|\Big)^2\\
&\le C^2\sum_{l=0}^\infty(1+l)^{2s}R_0^{-2l-2}\le C(\alpha,s),
\end{align*}
where $C(\alpha,s)$ depends on $\alpha$ and $s$. Therefore,  the proof is completed.
\end{proof}
\begin{lemma}\label{lem 2.6}
The operator on $\mathfrak{l}^s$ defined by $M$ is invertible.
\end{lemma}
\begin{proof}
Let $N\ge 1$ be a large integer to be chosen later. First we estimate the block of $(B)_{l,j}$, where $l\in [1,N]$, $j> N$:
\begin{equation*}
\sum_{l=1}^N\sum_{j> N}|B_{l,j}|\le \sum_{l=1}^{[N/2]+1}\sum_{2j> N}|B_{2l,2j}|+\sum_{l=1}^{[N/2]+1}\sum_{2j-1>N}|B_{2l-1,2j-1}|.
\end{equation*}
Using \eqref{eq 1.271}, the first summation on the right-hand side of the inequality above is bounded by
\begin{align}
\nonumber
&\sum_{2j>N}\sum_{k=1}^\infty \frac{2|\alpha|^k}{(kR_0)^{2j}} \sum_{l=1}^{[N/2]+1}\frac{{{2l+2j-1}\choose {2j-1}}}{(kR_0)^{2l}}\\
\nonumber
&< \sum_{2j>N}\sum_{k=1}^\infty\frac{2|\alpha|^k}
{(kR_0)^{2j}}\left(\sum_{l=0}^\infty
\frac{{{2l+2j-1}\choose {2j-1}}}{(kR_0)^{2l}}-1\right)\\
&=\sum_{2j>N}\sum_{k=1}^\infty|\alpha|^k
\left(\Big(\frac{1}{kR_0-1}\Big)^{2j}+\Big(\frac{1}{kR_0+1}\Big)^{2j}-2\Big(\frac{1}{kR_0}\Big)^{2j}\right)\nonumber\\
&\le CR_0^{-N-1},\label{eq 1.076}
\end{align}
where $C=C(\alpha)$.
 Similarly, by \eqref{eq 3.121} and \eqref{eq 4.32} we have
\begin{equation}
\sum_{l=1}^{[N/2]+1}\sum_{2j-1> N}|B_{2l-1,2j-1}|\le CR_0^{-N-1}.\label{eq 1.075}
\end{equation}
Therefore, combining \eqref{eq 1.076} and \eqref{eq 1.075} we obtain
\begin{equation}
\sum_{l=1}^N\sum_{j> N}|B_{l,j}|\le CR_0^{-N-1}.\label{eq 1.131}
\end{equation}

Next, we consider the $N$ dimensional matrix $Q:=\{M_{l,j}\}_{l,j=1}^N$. Since by Lemma \ref{lem2.3} $\{M_{l,j}\}_{l,j= 1}^\infty$ is diagonally dominant by column, $Q$ is diagonally dominant by column as well, which implies that $Q$ is invertible. We estimate $Q^{-1}$ as follows.
For odd number columns $ 2l-1\in[1, N]$, by \eqref{eq 4.32},
\begin{align*}
&|Q_{2l-1,2l-1}|-\sum_{\{j\,:\,2j-1\in[1,N], j\neq l\}}|Q_{2j-1,2l-1}|\ge |Q_{2l-1,2l-1}|-\sum_{j=1,j\neq l}^\infty|\tilde{M}_{2j-1,2l-1}|\\
&\ge 1-\sum_{k=1}^\infty\alpha^k \left(\Big(\frac{1}{kR_0-1}\Big)^{2l-1}
-\Big(\frac{1}{kR_0+1}\Big)^{2l-1}\right)
\ge C(\alpha,R_0)>0,
\end{align*}
where $C(\alpha,R_0)$ is a constant depending on $\alpha$ and $R_0$ but not on $N$.
For even number columns, by \eqref{eq 1.271} we obtain the same estimate,
\begin{equation*}
|Q_{2l,2l}|-\sum_{\{j\,:\,2j\in[2,N],j\neq l\}}|Q_{2j,2l}|\ge C(\alpha,R_0)>0.
\end{equation*}
Therefore, by \cite[Corollary 1]{Var75} we get
\begin{equation*}
\|Q^{-1}\|_{l_1}\le C(\alpha,R_0)^{-1},
\end{equation*}
which implies
\begin{equation}
\|Q^{-1}\|_{\mathfrak{l}^s}\le C(\alpha, R_0)^{-1}\sqrt{N}(1+N)^s,\label{eq 1.132}
\end{equation}
where $\|Q^{-1}\|_{\mathfrak{l}^s}$ is the operator norm of $Q^{-1}$ in
a finite dimensional subspace of $\mathfrak{l}^s$.
Indeed, for any $x=(x_1,x_2,\ldots,x_N)^T$, we compute
\begin{equation}
\|Q^{-1}x\|_{\mathfrak{l}^s}^2
=\sum_{l=1}^N\left(\sum_{j=1}^NQ^{-1}_{l,j}x_j\right)^2(1+l)^{2s}\le (1+N)^{2s}\sum_{l=1}^N\left(\sum_{j=1}^NQ_{l,j}^{-1}x_j\right)^2.\label{eq 3.125}
\end{equation}
Note that
\begin{equation*}
\sum_{l=1}^N\left(\sum_{j=1}^NQ^{-1}_{l,j}x_j\right)^2\le \left(\sum_{l=1}^N\Big|\sum_{j=1}^NQ_{l,j}^{-1}x_j\Big|\right)^2\le \left(\|Q^{-1}\|_{l_1}\sum_{j=1}^N|x_j|\right)^2.
\end{equation*}
By the Cauchy-Schwarz inequality, we have
\begin{align}
\left(\|Q^{-1}\|_{l_1}\sum_{j=1}^N|x_j|\right)^2&\le \|Q^{-1}\|_{l_1}^2\left(\sum_{l=1}^N(1+l)^{-2s}\right)
\|x\|^2_{\mathfrak{l}^s}\nonumber\\
&\le N\|Q^{-1}\|^2_{l_1}\|x\|^2_{\mathfrak{l}^s}.\label{eq 3.127}
\end{align}
Therefore, combining \eqref{eq 3.125}-\eqref{eq 3.127},  \eqref{eq 1.132} is proved.

Thanks to the estimates of the blocks, we are ready to prove the invertibility of $M$.
For any fixed $y\in \mathfrak{l}^s$, we need to find a unique $x\in \mathfrak{l}^s$ such that ${M}x=y$. Similar to $\mathfrak{l}^s$ we define the space
$$
\hat{\mathfrak{l}}^s=\Big\{a=(a_1,a_2,\cdots):\sum_{j=1}^\infty|a_j|^2(1+j)^{2s}<\infty \Big\}
$$
with the norm $\|a\|_{\hat{\mathfrak{l}}^s}=
\big(\sum_{j=1}^\infty|a_j|^2(1+j)^{2s}\big)^{1/2}.$ Notice that the summation index runs from $1$ instead of $0$ as in the definition of $\mathfrak{l}^s.$

By setting $x=({x}_0,\hat{x})^T$ and  $y=({y}_0,\hat{y})^T,$ where ${x}_0,{y}_0\in \bR$, $Mx=y$ is written as
\begin{equation*}
 \left(
 \begin{array}{cc}
 1& (B_{0,j})_{j\ge 1}\\
0&(M_{l,j})_{l,j\ge 1}
 \end{array}
 \right)
 \left(
 \begin{array}{c}
 {x}_0\\
 \hat{x}
 \end{array}
 \right)
 =\left(
 \begin{array}{c}
 {y}_0\\
 \hat{y}
 \end{array}
 \right).
\end{equation*}
Clearly, it is sufficient to prove that for any $\hat{y}\in \hat{\mathfrak{l}}^s$, there exists $\hat{x}\in \hat{\mathfrak{l}}^s$ such that $({M}_{l,j})_{l,j\ge 1}\hat{x}=\hat{y}.$
In fact, 
if there exists $\hat{x}\in \hat{\mathfrak{l}}^s$ so that $({M}_{l,j})_{l,j\ge 1}\hat{x}=\hat{y}$, it remains to set $x_0=y_0-(B_{0,j})_{j\ge 1}\hat{x}.$
From \eqref{eq 3.1210}, $x_0$ is well defined.
Obviously, $({M}_{l,j})_{l,j\ge 1}\hat{x}=\hat{y}$
can be further written as
\begin{equation*}
\left(
\begin{array}{cc}
Q& B_1\\
B_2& id+\hat{B}
\end{array}
\right)
\left(
\begin{array}{c}
x_1\\
x_2
\end{array}
\right)
=
\left(
\begin{array}{c}
y_1\\
y_2
\end{array}
\right),
\end{equation*}
where $\hat{x}=(x_1,x_2)^T, \hat{y}=(y_1,y_2)^T$, $x_1, y_1$ are $N$ dimensional vectors, $$
B_1=(B_{i,j})_{i\in [1,N], j> N},\quad
B_2=(B_{i,j})_{i> N, j\in [1,N]}, \quad
\hat{B}=(B_{i,j})_{i,j>N}.
$$
We rewrite the equation above as follows
\begin{equation*}
Qx_1+B_1x_2=y_1,\quad B_2x_1+x_2+\hat{B}x_2=y_2.
\end{equation*}
We take $Q^{-1}$ on the first equation to get
\begin{equation*}
x_1=-Q^{-1}B_1x_2+Q^{-1}y_1,\quad x_2=-B_2x_1-\hat{B}x_2+y_2.
\end{equation*}
It is sufficient to find a unique fixed point for the map on the right-hand side of the system above. We claim when $N$ is sufficiently large the map is a contraction, i.e., the operator norm of
\begin{equation*}
\left(
\begin{array}{cc}
0 & -Q^{-1}B_1\\
-B_2& -\hat{B}
\end{array}
\right)
\end{equation*}
is less than 1. Indeed, fix $\hat{x}=(x_1,x_2)^T\in \hat{\mathfrak{l}}^s$ with $\|\hat{x}\|_{\hat{\mathfrak{l}}^s}=1$, and we compute
\begin{align*}
&\|Q^{-1}B_1 x_2\|_{\hat{\mathfrak{l}}^s}^2+\|B_2 x_1+\hat{B}x_2\|_{\hat{\mathfrak{l}}^s}^2\\
&=\sum_{k=1}^N(1+k)^{2s}\left(\sum_{j=1}^NQ^{-1}_{k,j}
\Big(\sum_{l=N+1}^\infty B_{j,l}x_l\Big)\right)^2+\sum_{k> N}(1+k)^{2s}\left(\sum_{j=1}^\infty B_{k,j}x_j
\right)^2.
\end{align*}
By the definition of the $\mathfrak{l}^s$ norm, we have
\begin{align}\nonumber
&\sum_{k=1}^N(1+k)^{2s}\left(\sum_{j=1}^NQ^{-1}_{k,j}\Big(\sum_{l=N+1}^\infty B_{j,l}x_l\Big)\right)^2+\sum_{k> N}(1+k)^{2s}\Big(\sum_{j=1}^\infty B_{k,j}x_j\Big)^2\\\nonumber
&\le \|Q^{-1}\|^2_{\hat{\mathfrak{l}}^s}\left(\sum_{k=1}^N(1+k)^{2s}\Big(\sum_{l=N+1}^\infty B_{k,l}x_l\Big)^2\right)+\sum_{k>N}(1+k)^{2s}\Big(\sum_{j=1}^\infty |B_{k,j}|\Big)^2\\\label{eq 3.1211}
&\le \|Q^{-1}\|^2_{\hat{\mathfrak{l}}^s}\sum_{k=1}^N(1+k)^{2s}\sum_{l=N+1}^\infty B_{k,l}^2+\sum_{k>N}(1+k)^{2s}\Big(\sum_{j=1}^\infty |B_{k,j}|\Big)^2.
\end{align}

By \eqref{eq 1.131}, \eqref{eq 3.1210}, and the fact that $|B_{k,l}|<1$ for any $k, l\ge 1$, we know that
\begin{equation*}
\sum_{k=1}^N\sum_{l=N+1}^\infty B_{k,l}^2\le \sum_{k=1}^N\sum_{l=N+1}^\infty|B_{k,l}|\le CR_0^{-N-1},
\end{equation*}
and
\begin{equation*}
\sum_{j=1}^\infty |B_{k,j}|\le CR_0^{-k-1}.
\end{equation*}
Therefore, combining the two inequalities above with \eqref{eq 3.1211} and \eqref{eq 1.132}, we have
\begin{align*}
&\|Q^{-1}B_1x_2\|_{\hat{\mathfrak{l}}^s}^2+\|B_2x_1+\hat Bx_2\|_{\hat{\mathfrak{l}}^s}^2\\
&\le CN(1+N)^{4s}R_0^{-N-1}+C\sum_{k>N}(1+k)^{2s}R_0^{-2k-2}<1,
\end{align*}
provided that $N$ is sufficiently large only depending on $\alpha$, $R_0$, and $s$.  Hence $\tilde{M}$ and thus $M$ are one-to-one and onto, and $M$  is an invertible operator. Hence, we finish the proof.
\end{proof}

\begin{proof}[Proof of Proposition \ref{prop 1.291}]
From Lemmas \ref{lem 2.5} and \ref{lem 2.6}, $M$ induces a bounded and invertible map $T$ on $H^s_{\text{sym}}(\{|x|=R_0\})$.
For any $f\in H^s_{\text{sym}}(\{|x|=R_0\})$, let $T^{-1}(f)=g$ and suppose that $g=\sum_{j=0}^\infty a_je_j$. Since $T$ is bounded, we have
\begin{equation}\label{eq 4.33}
f=T\Big(\sum_{j=0}^\infty a_je_j\Big)=\sum_{j=0}^\infty a_jT(e_j).
\end{equation}
Therefore, $\{u_j=T(e_j)\}$ is a Schauder basis in $H^s_{\text{sym}}(\{|x|=R_0\})$. By the Parseval identity,
$$
\sum_{j=0}^\infty a_j^2 (1+j)^{2s}= C(s,R_0)\|g\|_{H^s(\{|x|=R_0\})}^2\le C(s,R_0)\|f\|_{H^s(\{|x|=R_0\})}^2,
$$
which gives \eqref{eq12.44}. The proposition is proved.
\end{proof}

\begin{remark}\label{rk 3.131}
In the previous proposition, we only consider functions even in $x_1$. For  functions odd in $x_1$, the same result can be proved and we only provide a sketch here. Define $\tilde{\Psi}_j$ as follows:
\begin{align*}
&\tilde{\Psi}_j(z)=\frac{2}{a_0+1}\sum_{k=1}^\infty(-\alpha)^k\frac{z^j}{(kz+\ri)^j}\quad \text{in}\,\, \{|z-\ri|<1\},\\
&\tilde{\Psi}_j(z)=(-1)^{(j+1)/2}\ri z^j+
\sum_{k=1}^\infty(-\alpha)^k\Big[\frac{z^j}{(kz+\ri)^j}-\frac{z^j}{(kz-\ri)^j}\Big]\\
&\quad \text{in}\,\, \{z:|z+\ri|>1\quad \text{and}\quad |z-\ri|>1\},\\
&\tilde{\Psi}_j(z)=-\frac{2}{a_0+1}\sum_{k=0}^\infty(-\alpha)^k\frac{z^j}{(kz-\ri)^j}\quad \text{in}\,\, \{|z+\ri|<1\}
\end{align*}
\underline{for $j$ odd}, and
\begin{align*}
&\tilde{\Psi}_j(z)=\frac{2}{a_0+1}\sum_{k=1}^\infty\alpha^k\frac{z^j}{(kz+\ri)^j}\quad \text{in}\,\, \{|z-\ri|<1\},\\
&\tilde{\Psi}_j(z)=(-1)^{j/2} z^j+\sum_{k=1}^\infty\alpha^k
\Big[\frac{z^j}{(kz+\ri)^j}+\frac{z^j}{(kz-\ri)^j}\Big]\\
&\quad \text{in}\,\, \{z:|z+\ri|>1\quad \text{and}\quad |z-\ri|>1\},\\
&\tilde{\Psi}_j(z)=\frac{2}{a_0+1}\sum_{k=0}^\infty\alpha^k\frac{z^j}{(kz-\ri)^j}\quad \text{in}\,\, \{|z+\ri|<1\}
\end{align*}
\underline{for $j$ even}.
Let $v(x)=R_0^{-j}\Im\tilde{\Psi}_j(z)$ and following the same argument in Proposition \ref{prop 1.291}, we can show that $\{v_j\}$ solves the equation \eqref{eq 2.221} and forms a Schauder basis of $H^s$ functions odd in $x_1$ on the circle $\{|x|=R_0\}$ provided that $R_0>2$.
\end{remark}

We finish this section by  stating \cite[Proposition 8.4]{LV00}.
\begin{proposition}\label{prop 1.152}
Given any multi-index $m$ there exists a constant $C_m$, independent of $j$ and $R_0$, so that the functions $u_j$ satisfy
\begin{equation*}
|D^mu_j(x_1,x_2)|\le C_mR_0^{-j}(j+|m|)^{|m|}
\end{equation*}
in each of the three regions $\overline{B_1}\cap \fB_1$, $\overline{B_1}\cap \fB_2$, and $\overline{B_1}\cap \fB_0$.
\end{proposition}

\section{Proof of Theorem \ref{thm 2.161}}
Now we are ready to prove our first main theorem.
\begin{proof}[Proof of Theorem \ref{thm 2.161}]
As explained in Remark \ref{rk 3.131}, we only need to consider $g$ even in $x_1$.
Without loss of generality, we may assume that $g$ is smooth on $\partial B_{R_0}=\{|x|=R_0\}$. If not, we may simply choose $2<R_0^\prime<R_0$  such that $K\subset B_{R_0^\prime}$. By the elliptic regularity, $u|_{\{|x|=R_0^\prime\}}$ is smooth. Therefore, we can replace $R_0$ by $R_0^\prime$.  From Proposition \ref{prop 1.291} we have
\begin{equation}\label{eq 4.35}
g=\sum_{j=0}^\infty g_ju_j,\quad \sum_{j=0}^\infty |g_j|^2(j+1)^{2s}<\infty
\end{equation}
for some $s>3/2$.
It is easily seen that
$$
\Big|\frac{d}{dz}\Psi_j\Big|\le C(j+1)
$$
in each of the three subdomains $\{z:|z+\ri|< 1\}$, $\{z:|z-\ri|< 1\}$, and $\{z: |z|\le R_0, |z+\ri|> 1, |z-\ri|> 1\}$, with constant $C$ depending on $R_0$. Therefore,
\begin{equation}\label{eq 4.34}
\|u_j\|_{H^1(B_{R_0})}\le C(j+1).
\end{equation}
We consider $U_k=\sum_{j=0}^kg_ju_j$. Let $h_k$ be the trace of $U_k$ on $\{|x|=R_0\}$, i.e., $h_k=U_k|_{\{|x|=R_0\}}$. By Proposition \ref{prop 1.291},  in particular \eqref{eq 4.33},  we have
$h_k\rightarrow g$ in $H^{1/2}(\{|x|=R_0\})$.

Moreover, by \eqref{eq 4.34} and \eqref{eq 4.35}, we get
\begin{align*}
&\sup_k\|U_k\|_{H^1(B_{R_0})}\le C\sum_{j=0}^\infty|g_j|(j+1)\\
&\le C\Big(\sum_{j=0}^\infty |g_j|^2(j+1)^{2s}\Big)^{1/2}
\Big(\sum_{j=0}^\infty(j+1)^{2-2s}\Big)^{1/2}<\infty.
\end{align*}
From the construction of $u_j$, $U_k$ is the solution of the following equation
\begin{equation*}
D_j(aD_jU_k)=0\quad\text{in}\,\, B_{R_0}, \quad  U_k=h_k\quad \text{on}\quad \partial B_{R_0}.
\end{equation*}
Since $h_k\rightarrow g$ in $H^{1/2}(\{|x|=R_0\})$, we have that $U_k\rightarrow u$ in $H^1(B_{R_0})$. Furthermore, for any multi-index $m$, by the interior elliptic estimates, we have the pointwise convergence
\begin{equation}
D^m U_k(x)\rightarrow D^m u(x)\label{eq 1.157}
\end{equation}
with $x\in B_{R_0}$ but not on $\{|x-(0,1)|=1\}$ and $\{|x-(0,-1)|=1\}$.
By Proposition \ref{prop 1.152}, we get
\begin{equation*}
|D^mu_j(x_1,x_2)|\le C_mR_0^{-j}(j+|m|)^{|m|}
\end{equation*}
in each of the three regions: $\overline{B_1}\cap \fB_1$, $\overline{B_1}\cap \fB_2$, and $\overline{B_1}\cap\fB_0$.
Therefore, by the Cauchy-Schwarz inequality,
\begin{align}\nonumber
&|D^m U_k(x)|\le \sum_{j=0}^k |g_j||D^m u_j(x)|\le C_m\sum_{j=0}^k|g_j|R_0^{-j}(j+|m|)^m\\\label{eq 1.156}
&\le C_m \Big(\sum_{j=0}^\infty |g_j|^2\Big)^{1/2}
\Big(\sum_{j=0}^\infty R_0^{-2j}(j+|m|)^{2|m|}\Big)^{1/2}\le C_m
\end{align}
for each multi-index $m$ in each of the three regions above. From \eqref{eq 1.157} and \eqref{eq 1.156}, it follows immediately $u$ has the desired smoothness in $\{|x|\le 1\}$.  In particular, $D^m u(x)$ has the same limit at the origin, whether we approach through the left cusp or through the right cusp. For $x\in K$ but outside $\{|x|\le 1\}$, the piecewise smoothness of $u$ follows from the classical elliptic regularity results; see, for instance, see \cite[Proposition 1.4]{LN03}. The theorem is proved.
\end{proof}

\section{Non-homogeneous equations with non-symmetric coefficients}

In this section, we consider non-homogeneous equations with non-symmetric coefficients
\begin{equation}
D_i(a(x)D_i u)=D_if_i,\label{eq 1.161}
\end{equation}
where $a(x)$ is equal to $a_0$ in $B_{r_1}(0,r_1)$, $b_0$ in $B_{r_2}(0,-r_2)$, and $1$ in $\bR^2\setminus(B_{r_1}(0,r_1)\cup B_{r_2}(0,-r_2))$,
and $a_0,b_0>0$.
The proof is divided into three steps. We shall first consider homogeneous equations with $r_1=r_2=1$ in Section \ref{sec4.1}, and then non-homogeneous equations with $r_1=r_2=1$ in Section \ref{sec4.2}, and finally the general case in Section \ref{sec4.3}.

\subsection{Homogeneous equations}
                                \label{sec4.1}
In this case, we basically adapt the proofs in Li and Vogelius \cite{LV00},  where they considered the special case $a_0=b_0>0$.

Recall that we use $\fB_1$ and $\fB_2$ to denote $B_1(0,1)$ and $B_1(0,-1)$, respectively, and $\fB_0:=\bR^2\setminus(\overline{\fB_1\cup \fB_2})$. Let $\mathcal{D}$ be an open bounded subset of $\bR^2$.
 The conformal mapping $z\rightarrow \ri/z$ maps $\fB_1$ to $\{\Re z>\.5\}$, $\fB_2$ to $\{\Re z<-\.5\}$, and $\fB_0$ to $\{\Re z\in (-\.5,\.5)\}$. This leads us to study the following homogeneous equation:
\begin{equation}\label{eq 3.141}
\tilde{L}u:=D_i(A(x)D_i u)=0,
\end{equation}
where
\begin{equation*}
A(x)=a_0\chi_{\{x_1>\.5\}}+\chi_{\{x_1\in(-\.5,\.5)\}}+b_0\chi_{\{x_1<-\.5\}}
\end{equation*}
Let
$$\alpha=\frac{a_0-1}{a_0+1},\quad\beta=\frac{b_0-1}{b_0+1}.$$
Choose a holomorphic function $\phi:\bC\setminus (0,0)\rightarrow \bC$ satisfying
\begin{equation}
                            \label{eq12.57}
\phi(\bar z)=\overline{\phi(z)}
\end{equation}
and
\begin{equation}
|\phi(z)|\le C\gamma^{|\Re z|},\quad \text{when}\quad |\Re z|>\.5,
\label{eq 3.171}
\end{equation}
where $0<\gamma<|\alpha\beta|^{-1}$ and $C$ is a constant.
We define $\Phi$ as follows:
\begin{align*}
&\Phi(z)=(1-\alpha)\sum_{k=0}^\infty(\alpha\beta)^k\Big[\phi(z+{2k})
-\beta\phi(-z-(2k+1))\Big]\quad \text{in}\,\, \Big\{\Re z>\.5\Big\},\\
&\Phi(z)=\phi(z)+\sum_{k=1}^\infty\Big[(\alpha\beta)^k
\big(\phi(z+{2k})+\phi(z-{2k})\big)\\
&\quad -(\alpha\beta)^{k-1}\big(\alpha\phi(-z+{2k-1})+\beta\phi(-z-({2k-1}))\big)\Big]\quad \text{in}\,\, \Big\{\Re z\in(-\.5,\.5)\Big\},\\
&\Phi(z)=(1-\beta)\sum_{k=0}^\infty(\alpha\beta)^k\Big[\phi(z-{2k})
-\alpha\phi(-z+{2k+1})\Big]\quad \text{in}\,\, \Big\{\Re z<-\.5\Big\}.
\end{align*}
Similar to \cite[Proposition 8.2]{LV00}, we have the following proposition.
\begin{proposition}\label{prop 3.142}
The function $u(x_1,x_2)=\Re \Phi(\ri/z)$ satisfies
\begin{equation}
D_i(a(x)D_iu)=0\quad \text{in}\,\, \bR^2.\label{eq 3.143}
\end{equation}
Moreover, $u$ is even in $x_1$.
\end{proposition}
\begin{proof}
The symmetry of $u$ in $x_1$ follows from \eqref{eq12.57}. By the property of the conformal mapping $z\rightarrow \ri/z$, it suffices to verify that $\Re\Phi(z)$ satisfies \eqref{eq 3.141}.
It is obvious that $\Re \Phi(z)$ is harmonic in each of the three strips $\{x_1<\.5\}$, $\{x_1\in (-\.5,\.5)\}$, and $\{x_1>\.5\}$. It remains to check the compatibility condition. Namely, $\Re \Phi(z)$ and $A(x)D_1\Re\Phi(z)$
are continues across the lines  $\{x_1=\.5\}$ and $\{x_1=-\.5\}$. Because $\Phi(z)$ is holomorphic, by the Cauchy-Riemann equation, it suffices to verify that $\Re \Phi(z)$ and $A(x)D_2\Im\Phi(z)$
are continues, which is equivalent to the continuities of
$\Re \Phi(z)$ and $A(x)\Im\Phi(z)$.
We only present the calculation associated with the continuities across the line $\{x_1=\.5\}$. The verification for the case $x_1=-\.5$ follows the same.

On one hand, we first compute
\begin{align*}
&2\Re \Phi(z)|_{x_1=\.5^-}\\
&=\phi(z)+\phi(\bar{z})+\sum_{k=1}^\infty\Big[(\alpha\beta)^k
\big(\phi(z+{2k})+\phi(\bar{z}+2k)
+\phi(z-{2k})\\
&\quad+\phi(\bar{z}-2k)\big)
-(\alpha\beta)^{k-1}\big(\alpha\phi(-z+{2k-1})+\alpha\phi(-\bar{z}+2k-1)\\
&\quad+\beta\phi(-z-{2k+1})+\beta\phi(-\bar{z}-2k+1)\big)\Big].
\end{align*}
Since $\bar{z}+z=1$, the right-hand side of the equality is equal to
\begin{align*}
&(1-\alpha)\sum_{k=0}^\infty(\alpha\beta)^k
\Big[\big(\phi(z+{2k})+\phi(\bar{z}+2k)\big)\\
&\quad -\beta\big(\phi(-z-(2k+1))
+\phi(-\bar{z}-(2k+1))\big)\Big],
\end{align*}
which is exactly equal to $2\Re\Phi(z)|_{x_1=\.5^+}$. Therefore,  the continuity of $\Re \Phi(z)$ across the line $\{x_1=\.5\}$ is proved.
It remains to check that $A(x)\Im\Phi$ is continuous and we do so by calculating
\begin{align*}
&2\ri A(x)\Im\Phi(z)|_{x_1=\.5^-}\\
&=\phi(z)-\phi(\bar{z})+\sum_{k=1}^\infty\Big[(\alpha\beta)^k
\big(\phi(z+{2k})-\phi(\bar{z}+2k)+\phi(z-{2k})-\phi(\bar{z}-2k)\big)\\
&\quad-(\alpha\beta)^{k-1}\big(\alpha\phi(-z+{2k-1})-\alpha\phi(-\bar{z}+2k-1)\\
&\quad+\beta\phi(-z-{2k+1})-\beta\phi(-\bar{z}-2k+1)\big)\Big]\\
&=(1+\alpha)\sum_{k=0}^\infty(\alpha\beta)^k
\Big[\phi(z+{2k})-\phi(\bar{z}+2k)\\
&\quad-\beta\phi(-z-(2k+1))
+\beta \phi(-\bar{z}-(2k+1))\Big]\\
&=A(x)\Im\Phi(z)|_{x_1=\.5^+}
\end{align*}
because $a_0(1-\alpha)=(1+\alpha)$.
This completes the proof of the proposition.
\end{proof}
When $\phi_j(z)=1/z^j$ for $j\ge 0$, which is holomorphic in $\bC\setminus(0,0)$ and satisfies \eqref{eq 3.171},  from Proposition \ref{prop 3.142}, $u_j:=R_0^{-j}\Re \Psi_j(z)$ is a solution to \eqref{eq 3.143} for each $j$ with
\begin{align*}
\Psi_j(z)&=(1-\alpha)\sum_{k=0}^\infty(\alpha\beta)^k
\Big[\frac{z^j}{(\ri+2kz)^j}+ \beta\frac{z^j}{(\ri+(2k+1)z)^j}\Big]\quad \text{in}\,\, \fB_1,\\
\Psi_j(z)&=(-1)^{(j+1)/2}\ri z^j+\sum_{k=1}^\infty(\alpha\beta)^k
\Big[\frac{z^j}{(\ri+2kz)^j}+\frac{z^j}{(\ri-2kz)^j}\\
&\quad -\alpha\frac{z^j}{(-\ri+(2k+1)z)^j}+\beta\frac{z^j}{(\ri+(2k+1)z)^j}\Big]\quad \text{in}\,\, \fB_0,\\
\Psi_j(z)&=(1-\beta)\sum_{k=0}^\infty(\alpha\beta)^k\Big[\frac{z^j}{(\ri-2kz)^j}
-\frac{z^j}{(-\ri+(2k+1)z)^j}\Big]\quad \text{in}\,\,\fB_2
\end{align*}
\underline{for $j$ odd}, and
\begin{align*}
\Psi_j(z)&=(1-\alpha)\sum_{k=0}^\infty(\alpha\beta)^k\Big[\frac{z^j}{(\ri+2kz)^j}
-\beta\frac{z^j}{(\ri+(2k+1)z)^j}\Big]\quad \text{in}\,\, \fB_1,\\
\Psi_j(z)&=(-1)^{j/2}z^j+\sum_{k=1}^\infty(\alpha\beta)^k
\Big[\frac{z^j}{(\ri+2kz)^j}+\frac{z^j}{(\ri-2kz)^j}\\
&\quad -\alpha\frac{z^j}{(-\ri+(2k+1)z)^j}-\beta\frac{z^j}{(\ri+(2k+1)z)^j}\Big]\quad \text{in}\,\, \fB_0,\\
\Psi_j(z)&=(1-\beta)\sum_{k=0}^\infty(\alpha\beta)^k\Big[\frac{z^j}{(\ri-2kz)^j}
-\alpha\frac{z^j}{(-\ri+(2k+1)z)^j}\Big]\quad \text{in}\,\, \fB_2
\end{align*}
\underline{for $j$ even}.
By straightforward calculations, we have the following propositions similar to \cite[Propositions 8.3 and 8.4]{LV00}.
\begin{proposition}
Given any integer $m$, there exists a constant $C_m$, independent of $j$, so that the functions $\Psi_j$ satisfy
$$
\Big|\Big(\frac{d}{dz}\Big)^m\Psi_j(z)\Big|\le C_m(j+m)^m
$$
in each of the three regions
$$
\Big\{z:|z|\le 1, |z+\ri|> 1, |z-\ri|> 1\Big\},
$$
$$
\Big\{z:|z|\le 1, |z-\ri|< 1\Big\},\quad \text{and}\quad
\Big\{z:|z|\le 1,|z+\ri|< 1\Big\}.$$
\end{proposition}
\begin{proposition}\label{prop 3.145}
Given any multi-index $m$, there exists a constant $C_m$ independent of $j$ and $R_0$, such that the functions $u_j$ satisfy
$$|D^m u_j(x_1,x_2)|\le C_mR_0^{-j}(j+|m|)^{|m|}$$
in each of the three regions: $\bar B_1\cap \fB_1$, $\bar B_1\cap \fB_2$, and $\bar B_1\cap\fB_0$.
\end{proposition}
Next, we investigate $u_j$ restricted on $\{|x|=R_0\}$ with $R_0>2$. By setting $z=R_0e^{\ri\theta}$ for $j\ge 0$,
\begin{align*}
u_{2j+1}&=(-1)^j\sin(2j+1)\theta\\
&\quad+\Re\Big\{\sum_{k=1}^\infty(\alpha\beta)^k
\Big[\frac{e^{\ri(2j+1)\theta}}{(\ri+2kR_0e^{\ri\theta})^{2j+1}}
+\frac{e^{\ri(2j+1)\theta}}{(\ri-2kR_0e^{\ri\theta})^{2j+1}}\\
&\quad-\alpha\frac{e^{\ri(2j+1)\theta}}{(-\ri+(2k+1)R_0e^{\ri\theta})^{2j+1}}
+\beta\frac{e^{\ri(2j+1)\theta}}{(\ri+(2k+1)R_0e^{\ri\theta})^{2j+1}}\Big]\Big\},\\
u_{2j}&=(-1)^j\cos(2j\theta)+\Re\Big\{\sum_{k=1}^\infty(\alpha\beta)^k
\Big[\frac{e^{\ri2j\theta}}{(\ri+2kR_0e^{\ri\theta})^{2j}}
+\frac{e^{\ri 2j\theta}}{(\ri-2kR_0e^{\ri\theta})^{2j}}\\
&\quad-\alpha\frac{e^{\ri 2j\theta}}
{(-\ri+(2k+1)R_0e^{\ri\theta})^{2j}}-\beta\frac{e^{\ri 2j\theta}}
{(\ri+(2k+1)R_0e^{\ri\theta})^{2j}}\Big]\Big\}.
\end{align*}
It is easy to see that
\begin{align*}
u_0&=\frac{(\alpha-1)(\beta-1)}{1-\alpha\beta},\\
u_{2j}&=(-1)^j\cos(2j\theta)+O(R_0^{-2j})\quad \text{on} \quad \{|x|=R_0\},\,\,j\ge 1,\\
u_{2j+1}&=(-1)^j\sin(2j+1)\theta+O(R_0^{-2j-1})\quad \text{on}\quad \{|x|=R_0\},\,\,\, j\ge 0.
\end{align*}
Therefore, similar to \cite[Propositions 8.5 and 8.6]{LV00}, we obtain the following denseness result on $\{u_j\}$.
\begin{proposition}\label{prop 3.144}
Given any $s\ge 0$, there exists a constant $C_s<\infty$, so that span$\{u_j|_{\{|x|=R_0\}}\}$ is dense in $H_{\text{sym}}^s(\{|x|=R_0\})$ provided that $R_0>C_s$. Moreover,  for any function $g\in H^s_{\text{sym}}(\{|x|=R_0\})$, we may approximate it by $g=\lim_{k}h_k$, with
\begin{equation*}
h_k=\sum_{j=0}^{N_k}\gamma_j^{(k)}u_j|_{\{|x|=R_0\}}\in \text{span}\big\{u_j|_{\{|x|=R_0\}}\big\}
\end{equation*}
and
\begin{equation*}
\Big(\sum_{j=0}^{N_k}|\gamma_j^{(k)}|^2(j+1)^{2s}\Big)^{1/2}\le C\|g\|_{H^s},
\end{equation*}
where $C$ depends on $s$ and $R_0$, but independent of $k$ and $g$, and $N_k\in \bN$.
\end{proposition}

\begin{remark}
In Proposition \ref{prop 3.144} the solutions are even in $x_1$. For the case  when solutions are odd in $x_1$, we define $\tilde{\Phi}$  as follows:
$$
\tilde{\Phi}(z)=(1-\alpha)
\sum_{k=0}^\infty(\alpha\beta)^k\Big[\phi(z+{2k})
+\beta\phi(-z-({2k+1}))\Big]\quad \text{in}\,\,\big\{\Re z>\.5\big\};
$$
\begin{align*}
\tilde\Phi(z)&=\phi(z)+\sum_{k=1}^\infty
\Big[(\alpha\beta)^k(\phi(z+{2k})+\phi(z-{2k}))\\
&\quad+(\alpha\beta)^{k-1}\big(
\alpha\phi(-z+{2k+1})+\beta\phi(-z-(2k+1))\big)\Big]\quad \text{in}\,\,
\big\{\Re z\in(-\.5,\.5)\big\};
\end{align*}
$$
\tilde\Phi(z)=(1-\beta)\sum_{k=0}^\infty(\alpha\beta)^k\Big[\phi(z-{2k})
+\alpha\phi(-z+{2k+1})\Big]\quad \text{in}\,\, \big\{\Re z<-\.5\big\},
$$
and $v(x_1,x_2)=\Im\tilde{\Phi}(\ri/z)$. It is easy to see that $v$ is odd in $x_1$ and following the proof of Proposition \ref{prop 3.142}, we can show that $v$ is a solution to \eqref{eq 3.143}. Further, we can prove similar result to Proposition \ref{prop 3.144}.
\end{remark}

Now we are in a good position to show  the following theorem, the proof of which is similar to that of Theorem \ref{thm 2.161}, and follows the lines in proving \cite[Proposition 8.1]{LV00}.
\begin{theorem}\label{thm 3.151}
Suppose $R_0>C_0$, where $C_0$ is a constant depending on $a_0$ and $b_0$. Let $g$ be in $H^{1/2}\big(\{|x|=R_0\}\big)$, and $u\in H^{1}(B_{R_0})$ denote the weak solution to
$$D_i(a(x)D_i u)=0\quad \text{in}\,\, B_{R_0},\quad u=g\quad \text{on}\quad \{|x|=R_0\}.$$
Then
$$
u\in C^{\infty}(K\setminus(\fB_1\cup \fB_2)),\quad u\in C^{\infty}(\overline{\fB}_1),\quad \text{and}\quad u\in C^\infty(\overline{\fB}_2)
$$
for any compact set $K\subset B_{R_0}$.
\end{theorem}
\begin{proof}
Without loss of generality, we assume that $g$ is smooth and even in $x_1$ as in the proof of Theorem \ref{thm 2.161}. Let $h_k$ be the approximating sequence of $g$ as in Proposition \ref{prop 3.144} with some fixed  $s>3/2$, i.e.,
$$h_k\rightarrow g\quad \text{in} \quad H^s\big(\{|x|=R_0\}\big) \quad \text{and}\quad h_k=\sum_{j=0}^{N_k} \gamma_j^{(k)}u_j|_{\{|x|=R_0\}},$$
with
\begin{equation*}
\Big(\sum_{j=0}^{N_k} |\gamma_j^{(k)}|^2(1+j)^{2s}\Big)^{1/2}\le C\|g\|_{H^s},
\end{equation*}
where $C$ depends on $R_0$.
By straightforward calculations, we have
$$
\Big|\frac{d}{dz}\Psi_j\Big|\le C(j+1)$$
in each of the three subdomains $\{|z+\ri|\le 1\},  \{|z-\ri|\le 1\}$, and $\{|z|\le R_0,|z+\ri|\ge 1,|z-\ri|\ge 1\}$, where $C$ depends on $R_0$. Hence,
$$\|u_j\|_{H^1(B_{R_0})}\le C(j+1).$$
By the Cauchy-Schwarz inequality, the sums $U_k=\sum_{j=0}^{N_k}\gamma_j^{(k)}u_j$ are convergent in $H^1(B_{R_0})$ and
\begin{align*}
&\|U_k\|_{H^1(B_{R_0})}\le C\sum_{j=0}^{N_k}|\gamma_j^{(k)}|(j+1)\\
&\le C\Big(\sum_{j=0}^{N_k}|\gamma_j^{(k)}|^2(j+1)^{2s}\Big)^{1/2}
\Big(\sum_{j=0}^{N_k}(j+1)^{2(1-s)}\Big)^{1/2}<\infty.
\end{align*}
By the linearity of the equation, $U_k$ is the solution to
\begin{equation*}
D_i(a(x)D_i U_k)=0\quad \text{in}\,\, B_{R_0},\quad U_k=h_k\quad \text{on}\quad \{|x|=R_0\}.
\end{equation*}
From our construction, $h_k\rightarrow g$ in $H^{1/2}\big(\{|x|=R_0\}\big)$. Thus we have $U_k\rightarrow u$ in  $H^1(B_{R_0})$.
From the elliptic regularity theory, we know that for any multi-index $m$,
\begin{equation}
D^m U_k(x)\rightarrow D^m u(x)\label{eq 3.147}
\end{equation}
at any point inside $B_{R_0}$, but not on the circles $\{|x-(0,\pm 1)|=1\}$.  From Proposition \ref{prop 3.145}, we get that for any multi-index $m$
\begin{align}\nonumber
&|D^m U_k(x)|\le \sum_{j=0}^{N_k}|\gamma_j^{(k)}||D^m u_j(x)|\le C_m\sum_{j=0}^{N_k}|\gamma_j^{(k)}|R_0^{-j}(j+|m|)^{|m|}\\\label{eq 3.146}
&\le C_m\Big(\sum_{j=0}^{N_k}|\gamma_j^{(k)}|^2\Big)^{1/2}
\Big(\sum_{j=0}^{N_k} R_0^{-2j}(j+|m|)^{2|m|}\Big)^{1/2}
\le C_m
\end{align}
in each of the three regions $\overline{B_1}\cap \fB_1$, $\overline{B_1}\cap \fB_2$, and $\overline{B_1}\cap \fB_0$.
From \eqref{eq 3.147} and \eqref{eq 3.146}, it follows immediately that $u$ has the desired smoothness properties in $\overline {B_1}$. For $x\in K$ but not in $\overline{B_1}$, the piecewise smoothness of $u$ follows from the classical elliptic regularity results.
\end{proof}

\subsection{Non-homogeneous equations}
                                \label{sec4.2}
In this subsection, we consider the non-homogeneous equations by constructing Green's function of the operator in \eqref{eq 1.161}.
By applying the conformal mapping $z\rightarrow \ri/z$, we shall first construct  Green's function of the operator $\tilde{L}$ defined in \eqref{eq 3.141}, i.e.,
\begin{equation*}
D_i(A(x)D_i \tG(x,y))=\delta(x-y),
\end{equation*}
where $D_i$ is with respect to $x_i$.

Denote ${\bf{k}}=(k,0)$, where $k\in \bZ$. Let $\overline{y}=(y_1,-y_2)$. It is well know that
$$
-\frac{1}{2\pi}\Delta\log|x-y|=\delta(x-y).
$$
In this section, for simplicity of exposition, we write $\Delta\log|x-y|=\delta(x-y)$.
We define $\tilde{G}(x,y)$ as follows:
\underline{when $y_1\in (-\.5,\.5)$},
\begin{align*}
\tilde{G}(x,y)&=(1-\alpha)\sum_{k=0}^\infty(\alpha\beta)^k
\Big[\log|x+{\bf2k}-y|\\
&\quad -\beta\log|x+{\bf 2k+1}+\overline{y}|\Big]
\quad \text{in}\,\, \big\{x_1>\.5\big\},\\
\tilde{G}(x,y)&=\log|x-y|+\sum_{k=1}^\infty\Big[(\alpha\beta)^k\big(\log|x+{\bf 2k}-y|+\log|x-{\bf 2k}-y|\big)\\
&\quad -(\alpha\beta)^{k-1}\big(\beta\log|x+{\bf 2k-1}+\overline{y}|+\alpha\log|x-{\bf (2k-1)}+\overline{y}|\big)\Big]\\
&\quad \text{in}\,\, \big\{x_1\in(-\.5,\.5)\big\},\\
\tilde{G}(x,y)&=(1-\beta)\sum_{k=0}^\infty(\alpha\beta)^k\Big[\log|x-{\bf2k}-y|\\
&\quad -\alpha\log|x-{\bf (2k+1)}+\overline{y}|\Big]
\quad \text{in}\,\, \big\{x_1<-\.5\big\};
\end{align*}
\underline{when $y_1>\.5$},
\begin{align*}
\tilde{G}(x,y)&=\log|x-y|+\alpha\log|x-{\bf 1}+\overline{y}|\\
&\quad -\frac{2\beta(1+\alpha)}{1+a_0}
\sum_{k=0}^\infty(\alpha\beta)^k\log|x+{\bf 2k+1}+\overline{y}|
\quad\text{in} \quad \big\{x_1>\.5\big\},\\
\tilde{G}(x,y)&=(1+\alpha)\sum_{k=0}^\infty(\alpha\beta)^{k}\Big[\log|x-{\bf 2k}-y|\\
&\quad -\beta\log|x+{\bf 2k+1}+\overline{y}|\Big]
\quad\text{in}\,\, \big\{x_1\in(-\.5,\.5)\big\}, \\
\tilde{G}(x,y)&=\frac{2(1+\alpha)}{1+b_0}\sum_{k=0}^\infty
(\alpha\beta)^{k}\log|x-{\bf 2k}-y|\quad \text{in}\,\, \big\{x_1<-\.5\big\};
\end{align*}
\underline{when $y_1<-\.5$},
\begin{align*}
\tilde{G}(x,y)&=\frac{2(1+\beta)}{1+a_0}\sum_{k=0}^\infty (\alpha\beta)^k\log|x+{\bf 2k}-y|\quad \text{in}\,\, \big\{x_1>\.5\big\},\\
\tilde{G}(x,y)&=(1+\beta)\sum_{k=0}^\infty(\alpha\beta)^k\Big[\log|x+{\bf 2k}-y|\\
&\quad -\alpha\log|x-{\bf(2k+1)}+\overline{y}|\Big]
\quad \text{in}\,\, \big\{x_1\in(-\.5,\.5)\big\},\\
\tilde{G}(x,y)&=\log|x-y|+\beta\log|x+{\bf 1}+\overline{y}|\\
&\quad -\frac{2\alpha(1+\beta)}{1+b_0}\sum_{k=0}^\infty(\alpha\beta)^{k}
\log|x-{\bf (2k+1)}+\overline{y}|
\quad \text{in}\,\, \big\{x_1<-\.5\big\}.
\end{align*}

\begin{proposition}\label{prop 2.22}
The function $\tilde{G}(x,y)$ defined above is Green's function of $\tilde L$, i.e.,
$$D_i(A(x)D_i\tilde{G}(x,y))=\delta(x-y).$$
\end{proposition}
\begin{proof}
To show $\tilde{G}(x,y)$ is Green's function, it is sufficient to prove that for $y\in \bR^2$ a.e., $\Delta \tilde{G}(x,y)=\delta(x-y)$ for $x\notin \{x_1=\.5\}\cup \{x_1=-\.5\}$ and $\tilde{G}(x,y), A(x)D_1\tilde{G}(x,y)$ are continuous in $x$ across the two lines $\{x_1=\pm\.5\}$.

We first consider the case when $y_1\in (-\frac{1}{2},\frac{1}{2})$.  It is obvious that for $x_1<-\frac{1}{2}$ and $k\ge 0$ $$|x-{\bf{k}}-y|>0,\quad |x-{\bf{k}}+\overline{y}|>0,$$ which implies
$\Delta_x\tilde{G}(x,y)=0$.
Similarly, we can check that for $x_1>\frac 1 2$,
$\Delta_x\tilde{G}(x,y)=0$.
Moreover, for $x_1\in(-1/2,1/2)$, note that
$$\Delta_x\log|x-y|=\delta(x-y),$$
and
\begin{align*}
&\Delta_x\log|x+{\bf{2k}}-y|=\Delta_x\log|x-{\bf{2k}}-y|\\
&=\Delta_x\log|x+{\bf 2k-1}+\overline{y}|=\Delta_x\log|x-{\bf (2k-1)}+\overline{y}|=0
\end{align*}
provided that $k\ge 1$. Thus,
$$\Delta_x\tilde{G}(x,y)=\delta(x-y)\quad \text{for}\,\, x\in (-\frac{1}{2},\frac{1}{2}).$$
It remains to verify the continuities of $\tilde{G}(x,y)$ and $A(x)D_1\tilde{G}(x,y)$ across the lines $\{x_1=1/2\}$. For simplicity, we only present the calculations associated with the case $x_1=1/2$.
We first check that $\tilde{G}(x,y)$ is continuous at $x_1=1/2$. By a straightforward calculation, we have
\begin{align*}
&\tilde{G}(x,y)|_{x_1=\.5^-}=\frac{1}{2}\log\Big((\frac{1}{2}-y_1)^2+(x_2-y_2)^2\Big)\\
&\quad+\frac 1 2\sum_{k=1}^\infty\Big\{(\alpha\beta)^{k}
\Big[\log\big((2k+\frac{1}{2}-y_1)^2+(x_2-y_2)^2\big)\\
&\quad+\log\big((2k-\frac{1}{2}+y_1)^2+(x_2-y_2)^2\big)\Big]\\
&\quad-(\alpha\beta)^{k-1}
\Big[\beta\log\big((2k-1+\frac{1}{2}+y_1)^2+(x_2-y_2)^2\big)\\
&\quad+\alpha\log\big((2k-1-\.5-y_1)^2+(x_2-y_2)^2\big)\Big]\Big\}\\
&=\frac {1-\alpha} 2\sum_{k=0}^\infty(\alpha\beta)^{k}
\Big[\log\Big((2k+\frac{1}{2}-y_1)^2+(x_2-y_2)^2\Big)\\
&\quad-\beta
\log\Big((2k+\frac{3}{2}+y_1)^2+(x_2-y_2)^2\Big)\Big]\\
&=\tG(x,y)|_{x_1=\frac{1}{2}^+}.
\end{align*}
Next we check that $A(x)D_1\tilde{G}(x,y)$ is continuous across $\{x_1=1/2\}$.  We compute
\begin{align*}
&A(x)D_1\tilde{G}(x,y)|_{x_1=\frac{1}{2}^-}=D_1\log|x-y||_{x_1=\frac{1}{2}}\\
&\quad+\sum_{k=1}^\infty(\alpha\beta)^{k}
\Big[D_1\log|{\bf{2k}}+x-y|_{x_1=\frac{1}{2}}
+D_1\log|{\bf 2k}-x+y|_{x_1=\frac{1}{2}}\\
&\quad-\beta D_1\log|{\bf 2k-1}+x+\overline{y}|_{x_1=\frac{1}{2}}-\alpha D_1\log|{\bf 2k-1}-x-\overline{y}|_{x_1=\frac{1}{2}}\Big]\\
&=(1+\alpha)\sum_{k=0}^\infty(\alpha\beta)^{k}
\Big[\frac{2k+\.5-y_1}{(2k+\.5-y_1)^2+(x_2-y_2)^2}-\beta\frac{2k+\frac 3 2+y_1}{(2k+\frac 3 2+y_1)^2+(x_2-y_2)^2}\Big].
\end{align*}
On the other hand, we have
\begin{align*}
&A(x)D_1\tilde{G}(x,y)|_{x_1=\.5^+}=a_0D_1\tilde{G}(x,y)|_{x_1=\.5^+}\\
&=a_0(1-\alpha)\sum_{k=0}^\infty(\alpha\beta)^{k}\Big[\frac{2k+\.5-y_1}{(2k+\.5-y_1)^2+(x_2-y_2)^2}\\
&\quad -\beta\frac{2k+\frac 3 2+y_1}{(2k+\frac 3 2+y_1)^2+(x_2-y_2)^2}\Big].
\end{align*}
Since
$$a_0(1-\alpha)=(1+\alpha),$$
we get
$$A(x)D_1\tilde{G}(x,y)|_{x_1=\.5^-}=A(x)D_1\tilde{G}(x,y)|_{x_1=\.5^+}.$$

When $y_1>1/2$, the singularity appears in the region $\{x_1>1/2\}$. For completeness, we present the calculations below. We first verify
$$\Delta_x\tilde{G}(x,y)=\delta(x-y).$$
For $x_1>1/2$ and $k\ge -1$, it is easy to see that $|{\bf{k}}+x+\overline{y}|>0$, which implies
$$\Delta_x\log|x+{\bf{k}}+\overline{y}|=0.$$
Combining with the fact that
$$\Delta_x\log|x-y|=\delta(x-y),$$
 we get
$$
\Delta_x\tilde{G}(x,y)=\delta(x-y)\quad \text{for} \quad x_1>1/2.
$$
Similarly, for $x_1\in(-1/2,1/2)$,
$$
|x-y|>0,\quad |x+{\bf2k-1}+y|>0,\quad\text{and} \quad |x-{\bf2k}-y|>0
$$
provided that $k\ge 1$. Thus,
$$\Delta_x\tilde{G}(x,y)=0\quad \text{for}\,\, x_1\in(-1/2,1/2).$$
In the same way, we have
$$\Delta_x\tilde{G}(x,y)=0\quad \text{for}\,\, x_1<-1/2.$$
Next we verify the continuities of $\tilde{G}(x,y)$ and $A(x)D_1\tilde{G}(x,y)$ at $x_1=\pm 1/2$. By the same argument as in the case $y_1\in(-1/2,1/2)$, without loss of generality, we only check the continuities at $x_1=1/2$. To this end, we compute
\begin{align*}
\tilde{G}(x,y)|_{x_1=\frac{1}{2}^-}
&=\frac {1+\alpha} 2\sum_{k=0}^\infty(\alpha\beta)^k
\Big[\log\Big((2k-\.5+y_1)^2+(x_2-y_2)^2\Big)\\
&\quad-\beta\log\Big((2k+\frac 3 2+y_1)^2+(x_2-y_2)^2\Big]\\
&=\frac{1+\alpha}{2}\log\Big((\.5-y_1)^2+(x_2-y_2)^2\Big)\\
&\quad+(1+\alpha)(\alpha\beta-\beta)\sum_{k=0}^\infty
\frac{(\alpha\beta)^{k}}{2}\log\Big((2k+\frac 3 2+y_1)^2+(x_2-y_2)^2\Big).
\end{align*}
On the other side of $x_1=1/2$, we calculate
\begin{align*}
\tilde{G}(x,y)|_{x_1=\.5^+}&=\frac{1+\alpha}{2}\log\Big((\.5-y_1)^2+(x_2-y_2)^2\Big)\\
&\quad-\frac{2\beta(1+\alpha)}{1+a_0}\sum_{k=0}^\infty
\frac{(\alpha\beta)^{k}}{2}\log\Big((\.5+y_1+2k+1)^2+(x_2-y_2)^2\Big).
\end{align*}
Since
$$(1+\alpha)(\alpha\beta-\beta)=-\frac{2\beta(1+\alpha)}{1+a_0},$$
it follows immediately that
\begin{equation*}
\tilde{G}(x,y)|_{x_1=\.5^-}=\tilde{G}(x,y)|_{x_1=\.5^+}.
\end{equation*}

Next, we verify that $A(x)D_1\tilde{G}(x,y)$ is continuous at $x_1=1/2$ and compute
\begin{align*}
&A(x)D_1\tilde{G}(x,y)|_{x_1=\.5^-}=(1+\alpha)\sum_{k=0}^\infty
(\alpha\beta)^{k}\Big[D_1\log|{\bf 2k}-x+y||_{x_1=\.5}\\
&\quad-\beta D_1\log|{\bf 2k+1}+x+\overline{y}||_{x_1=\.5}\Big]\\
&=-(1+\alpha)\sum_{k=0}^\infty(\alpha\beta)^{k}\Big[\frac{2k-\.5+y_1}
{(2k-\.5+y_1)^2+(x_2-y_2)^2}\\
&\quad-\beta\frac{2k+\frac 3 2+y_1}{(2k+\frac 3 2+y_1)^2+(x_2-y_2)^2}\Big]\\
&=-(1+\alpha)\frac{y_1-\.5}{(y_1-\.5)^2+(x_2-y_2)^2}\\
&\quad-(1+\alpha)^2\beta\sum_{k=0}^\infty(\alpha\beta)^{k}
\frac{2k+y_1+\frac 3 2}{(2k+\frac 3 2+y_1)^2+(x_2-y_2)^2}.
\end{align*}
On the other hand, we have
\begin{align*}
&A(x)D_1\tilde{G}(x,y)|_{x_1=\.5^+}=-a_0(1-\alpha)\frac{y_1-\.5}{(\.5-y_1)^2+(x_2-y_2)^2}\\
&\quad-\frac{2a_0\beta(1+\alpha)}{1+a_0}
\sum_{k=0}^\infty(\alpha\beta)^{k}\frac{\.5+y_1+2k+1}
{(\.5+y_1+2k+1)^2+(x_2-y_2)^2}.
\end{align*}
Because
$$a_0(1-\alpha)=(1+\alpha)\quad \text{and}\quad (1+\alpha)^2\beta=\frac{2a_0\beta(1+\alpha)}{1+a_0},$$
it follows that
$$A(x)D_1\tilde{G}(x,y)|_{x_1=\.5^-}=A(x)D_1\tilde{G}(x,y)|_{x_1=\.5^+}.$$
Therefore, $\tilde{G}(x,y)$ satisfies
\begin{equation}
D_i(A(x)D_i\tilde{G}(x,y))=\delta(x-y)\label{eq 3.151}
\end{equation}
for the case $y_1>1/2$.
Similarly, we can check that \eqref{eq 3.151} holds when $y_1<-1/2$ as well. The details are omitted.
Thus,  $\tilde{G}(x,y)$ is  Green's function to the divergence type operator $\tilde L$ and we complete the proof of the proposition.
\end{proof}

Now, let us turn back to the original operator $Lu=D_i(a(x)D_iu(x))$.  The conformal mapping $z\rightarrow \ri/z$ can be written in real variables as $\Theta:\bR^2\rightarrow \bR^2$:
\begin{equation*}
\Theta_1(x)=\frac{x_2}{x_1^2+x_2^2},\quad
\Theta_2(x)=\frac{x_1}{x_1^2+x_2^2}.
\end{equation*}

For any integer $k$, denote $X_k(x)=\Theta(\Theta(x)+k)$, which is a conformal map.
According to $\tilde G$, we define $G (x,y)$ as follows:
\underline{when $y\in \fB_0$},
\begin{align*}
G (x,y)&=(1-\alpha)\sum_{k=0}^\infty(\alpha\beta)^k\big[\log|X_{2k}(x)-y|
-\beta\log|X_{2k+1}(x)-\overline{y}|\big]
\quad \text{for}\,\, x\in \fB_1,\\
G (x,y)&=\log|x-y|+\sum_{k=1}^\infty(\alpha\beta)^k\Big[
\log|X_{2k}(x)-y|+\log|X_{-2k}(x)-y|\Big)\\
&\quad-\beta\log|X_{2k-1}(x)-\overline{y}|
-\alpha\log|X_{-(2k-1)}(x)-\overline{y}|\Big]\quad \text{for}\,\, x\in\fB_0,\\
G (x,y)&=(1-\beta)\sum_{k=0}^\infty(\alpha\beta)^k
\big[\log|X_{-2k}(x)-y|\\
&\quad-\alpha\log|X_{-(2k+1)}(x)-\overline{y}|\big]
\quad \text{for}\,\, x_1\in \fB_2;
\end{align*}
\underline{when $y\in \fB_1$},
\begin{align*}
G (x,y)&=\log|x-y|+\alpha\log|X_{-1}(x)-\overline{y}|\\
&\quad-\frac{2\beta(1+\alpha)}{1+a_0}\sum_{k=0}^\infty(\alpha\beta)^k\log|X_{2k+1}(x)-\overline{y}|
\quad\text{for} \quad x\in \fB_1,\\
G (x,y)&=(1+\alpha)\sum_{k=0}^\infty(\alpha\beta)^{k}\big[\log|X_{-2k}(x)-y|
-\beta\log|X_{2k+1}(x)-\overline{y}|\big]\quad
\text{for}\,\, x\in \fB_0, \\
G (x,y)&=\frac{2(1+\alpha)}{1+b_0}\sum_{k=0}^\infty(\alpha\beta)^{k}
\log|X_{-2k}(x)-y|\quad \text{for}\,\, x\in \fB_2;
\end{align*}
\underline{when $y\in \fB_2$},
\begin{align*}
G (x,y)&=\frac{2(1+\beta)}{1+a_0}\sum_{k=0}^\infty (\alpha\beta)^k\log|X_{2k}(x)-y|\quad \text{for}\,\, x_1\in \fB_1,\\
G (x,y)&=(1+\beta)\sum_{k=0}^\infty(\alpha\beta)^k\Big[\log|X_{2k}(x)-y|\\ &\quad- \alpha\log|X_{-(2k+1)}(x)-\overline{y}|\Big]
\quad \text{for}\,\, x\in \fB_0,\\
G (x,y)&=\log|x-y|+\beta\log|X_1(x)-\overline{y}|\\
&\quad-\frac{2\alpha(1+\beta)}{1+b_0}\sum_{k=0}^\infty(\alpha\beta)^{k}
\log|X_{-2k-1}(x)-\overline{y}|
\quad \text{for}\,\, x\in \fB_2.
\end{align*}

\begin{proposition}
The function G(x,y) defined above is Green's function of $L$.
\end{proposition}
\begin{proof}
Similar to the verification of $\tG(x,y)$ being Green's function of $\tilde{L}$, we first check that
$$\Delta_xG (x,y)=\delta(x-y)\quad \text{for}\,\,
x\notin \partial (\fB_1\cup \fB_2).
$$
In order to show this, we consider the case when $y\in\fB_0$ as an example. When $x\in \fB_1\cup \fB_2$, we show that $G (x,y)$ is harmonic. For instance, when $x\in \fB_1$,
$$
\Theta(x)+{\bf 2k}\in \big\{x_1>1/2\big\}\quad \text{and}\quad
\Theta(y)\in \big\{x_1\in (-1/2,1/2)\big\}.
$$
Therefore, $\Theta(x)+{\bf 2k}\neq\Theta(y)$, implying $X_{2k}(x)\neq y$. Similarly,  we have $
X_{2k-1}(x)\neq\overline{y}$.
Combining with the facts that $\Delta_x\log|x-y|=0$ when $x\neq y$ in $\bR^2$ and that $X_k$ is conformal, we obtain that $G (x,y)$ is harmonic in $\fB_1$. In the same way, we can show that $G (x,y)$ is harmonic in $\fB_2$ as well. When $x\in \fB_0$,  as we mentioned in the beginning of this section, we use the notation $\Delta_x\log|x-y|=\delta(x-y)$. Each term in the expression of $G (x,y)$, with the  exception of  $\log|x-y|$, is harmonic in $\fB_0$  by the same argument in proving  $\Delta_xG (x,y)=0$ in $\fB_1$.  Hence, when $x\in \fB_0$, $\Delta_xG (x,y)=\delta(x-y)$.
For the case when $y\in \fB_1\cup \fB_2$, the same argument can be implemented to show that $\Delta_xG (x,y)=\delta(x-y)$, and we omit the details.

It remains to verify the continuities of $G (x,y)$ and $a(x)D_\nu G (x,y)$  across the two circles $\{|x-(0,\pm 1)=1|\}$, where $\nu$ is the unit normal vector field of $\partial (\fB_1\cup \fB_2)$.

Because $\Theta$ is a conformal map, the continuities of $G(x,y)$ and $a(x)D_\nu G (x,y)$ is equivalent to the continuities of  $G(\Theta(x),y)$ and $a(\Theta(x))D_1(G (\Theta(x),y))$, respectively. Note that $a(\Theta(x))=A(x)$ and
$$\frac{|x|}{|\Theta(y)|}=\frac{|y|}{|\Theta(x)|},$$
which by similarity of triangles implies that
\begin{equation}
\frac{|\Theta(x)-y|}{|x-\Theta(y)|}=\frac{|y|}{|x|}.\label{eq 2.101}
\end{equation}
We take the case when $y\in \fB_1$ as an example and the other cases can be verified in the same way. Since $\Theta(\Theta(x))=x$, we have
\begin{align*}
&G (\Theta(x),y)=\log|\Theta(x)-y|+\alpha\log|\Theta(x-{\bf 1})-\overline{y}|\\
&\quad-\frac{2\beta(1+\alpha)}{1+a_0}\sum_{k=0}^\infty(\alpha\beta)^k\log|\Theta(x+{\bf2k+1})-\overline{y}|
\quad\text{in} \quad \big\{x_1>\.5\big\},\\
&G (\Theta(x),y)=(1+\alpha)\sum_{k=0}^\infty(\alpha\beta)^{k}
\Big[\log|\Theta(x-{\bf 2k})-y|\\
&\quad-\beta\log|\Theta(x+{\bf 2k+1})-\overline{y}|\Big]\quad
\text{in}\,\, \big\{x_1\in(-\.5,\.5)\big\}, \\
&G(\Theta(x),y)=\frac{2(1+\alpha)}{1+b_0}
\sum_{k=0}^\infty(\alpha\beta)^{k}\log|\Theta(x-{\bf 2k})-y|\quad \text{in}\,\, \big\{x_1<-\.5\big\}.
\end{align*}
By taking \eqref{eq 2.101} into account, $G (\Theta(x),y)$ has the expression: in $\big\{x_1>\.5\big\}$
\begin{align*}
&G (\Theta(x),y)=\log|x-\Theta(y)|+\alpha\log|x-{\bf 1}+\overline{\Theta(y)}|\\
&\quad-\frac{2\beta(1+\alpha)}{1+a_0}\sum_{k=0}^\infty(\alpha\beta)^k\log|x+{\bf 2k+1}+\overline{\Theta(y)}|-\log|x|-\alpha\log|x-{\bf 1}|\\
&\quad+\frac{2\beta(1+\alpha)}{1+a_0}\sum_{k=0}^\infty(\alpha\beta)^k\log|x+{\bf 2k+1}|+\frac{(1-\beta)(1+\alpha)}{1-\alpha\beta}\log|y|;
\end{align*}
in $\big\{x_1\in (-\.5,\.5)\big\}$
\begin{align*}
&G (\Theta(x),y)=(1+\alpha)\sum_{k=0}^\infty(\alpha\beta)^k\Big[\log|x-{\bf 2k}-\Theta(y)|-\beta
\log|x+{\bf 2k+1}+\overline{\Theta(y)}|\Big]\\
&\quad-(1+\alpha)\sum_{k=0}^\infty(\alpha\beta)^k\Big[\log|x-{\bf 2k}|-\beta\log|x+{\bf 2k+1}|\Big]+\frac{(1+\alpha)(1-\beta)}{1-\alpha\beta}\log|y|;
\end{align*}
and in $\big\{x_1<-\.5\big\}$
\begin{align*}
&G (\Theta(x),y)=\frac{2(1+\alpha)}{1+b_0}\sum_{k=0}^{\infty}
(\alpha\beta)^k\log|x-{\bf 2k}-\Theta(y)|\\
&\quad-\frac{2(1+\alpha)}{1+b_0}\sum_{k=0}^{\infty}(\alpha\beta)^k\log|x-{\bf 2k}|+\frac{(1-\beta)(1+\alpha)}{1-\alpha\beta}\log|y|.
\end{align*}
Observe that $$G (\Theta(x),y)=\tG(x,\Theta(y))-H(x)+\frac{(1+\alpha)(1-\beta)}{1-\alpha\beta}\log|y|,$$ where $\tG(x,y)$ is Green's function of $\tilde{L}$ and $H(x)$ is a function obtained by replacing $y$ with $0$ in the expression of $\tG(x,y)$. Since we verify the continuities of $\tG(x,y)$ and $A(x)D_1\tG(x,y)$ across the lines $\{x_1=\pm \.5\}$ in Proposition \ref{prop 2.22}, the same proof shows that $H(x)$ and $A(x)D_1H(x)$ are continuous across the lines $\{x_1=\pm \.5\}$. Therefore, $G (\Theta(x),y)$ and $A(x)D_1G (\Theta(x),y)$ are continuous by linearity, and $G (x,y)$ is the desired Green's function.
\end{proof}
With the help of Green's function constructed above, we are ready to consider the non-homogeneous equation
\begin{equation*}
D_i(a(x)D_iu(x))=D_if_i(x)
\end{equation*}
in general $\mathcal{D}\subset \bR^2$.
Now, we state our theorem in the case when $r_1=r_2=1$.
\begin{theorem}\label{thm 2.191}
Let $n\in \mathbb{N}\cup \{0\}$ and $\gamma\in(0,1)$. Assume that $u$ is a weak solution to the  equation
\begin{equation*}
D_i(a(x)D_iu(x))=D_if_i\quad \text{in}\,\, \mathcal{D},
\end{equation*}
where
\begin{align*}
&a(x)=a_0\quad \text{in}\,\, \fB_1,\quad
a(x)=b_0\quad \text{in}\,\, \fB_2,\quad
a(x)=1\quad \text{in}\,\, \fB_0,
\end{align*}
and for each $i$, $f_i$ is piecewise $C^{n,\gamma}$, i.e.,
\begin{align*}
f_i\in C^{n,\gamma}(\fB_1),\quad f_i\in C^{n,\gamma}(\fB_2),\quad
f_i\in C^{n,\gamma}(\fB_0).
\end{align*}
 Let $\mathcal{D}_\varepsilon=\{x\in\mathcal{D},\text{dist}(x,\partial \mathcal{D})\ge \varepsilon\}$ for  $\varepsilon>0$.
Then $u$ is piecewise $C^{n+1,\gamma}$ in $\mathcal{D}_{\varepsilon}$ up to the boundary, i.e.,
$$
u\in C^{n+1,\gamma}(\mathcal{D}_\varepsilon\cap {\fB_1}),\,\, u\in C^{n+1,\gamma}(\mathcal{D}_\varepsilon\cap{\fB_2}),\quad u\in C^{n+1,\gamma}(\mathcal{D}_\varepsilon \cap\fB_0).
$$
\end{theorem}
\begin{proof}
We prove the theorem by considering two cases.

{\bf Case 1:} $(0,0)\notin\mathcal{D}$.  Define $\Omega_1=\mathcal{D}\cap \fB_1, \Omega_2=\mathcal{D}\cap \fB_2$, and $\Omega_0=\mathcal{D}\cap\fB_0$. Because $(0,0)\notin \mathcal{D}$, it is easy to see that any point in $\mathcal{D}$ belongs to at most two subdomains $\overline{\Omega_i}$, which is exactly the case in  \cite[Remark 3(ii)]{DongARMA12}.  Therefore, we apply \cite[Theorem 2 and Remark 3(ii)]{DongARMA12} to obtain that when $n=0$, $u\in C^{1,\gamma}$ piecewise in $\mathcal{D}_{\varepsilon}$, i.e., for any $0\le i\le 2$,
$$u\in C^{1,\gamma}(\mathcal{D}_{\varepsilon}\cap \Omega_i).$$

For $n>0$, we  use an induction argument. For a ball  away from the circles $\{|x-(0,\pm 1)|=1\}$, the conclusion follows from the classical Schauder estimate for Poisson's equation. We only need to consider a ball $B_l(x)\subset\mathcal{D}_\varepsilon$ and $x\in \{|x-(0,\pm 1)|=1\}$.  Notice that  by locally flattening the boundary, it is sufficient to consider
\begin{equation*}
D_i(a_{ij}D_j u)=D_ih_i,
\end{equation*}
where  $a_{ij}$ are piecewise smooth in $\bR^2_+$ and $\bR^2_-$ with bounded derivatives and $h_i\in C^{n,\gamma}(\bR^2_+)$, $h_i\in C^{n,\gamma}(\bR^2_-)$, where $\bR^2_+$ (or $\bR^2_-$) is the set of points on the plane such that $x_2>0$ (or $x_2<0$).
Here we only give a sketch of the proof. For $n=1$, by taking derivative with respect to the tangential variable $x_1$, we have
\begin{equation*}
D_i(a_{ij}D_jD_1u)=D_iD_1h_i-D_i(D_1(a_{ij})D_ju).
\end{equation*}
Thanks to the case $n=0$, we have that $Du\in C^{\gamma}$ piecewise, which implies the right-hand side can be written as $D_i\mathfrak{f}_i$ for some piecewise $C^\gamma$ functions $\mathfrak{f}_i$. Therefore, we apply the result of the case $k=0$ to obtain that $DD_1u\in C^{\gamma}$ piecewise.
It remains to estimate $D^2_2 u$, which is obtained from the  formula
\begin{equation*}
 D_2^2u=a_{22}^{-1}(D_ih_i-D_1(a_{1j}D_ju)-D_2(a_{21}D_1u)).
\end{equation*}
Therefore,  $DD_2 u$ is piecewise $C^{1,\gamma}$ as well. By induction, it is easy to prove that $u\in C^{n+1,\gamma}$ piecewise.

{\bf Case 2:} $(0,0)\in \mathcal{D}$. There exists $0<l<1$ such that $B_l\subset \mathcal{D}$. We define a cutoff function $\eta\in C_0^\infty(B_l)$, which equals 1 on $B_{l/2}$.
Let $v=u\eta$, which satisfies
\begin{equation}
D_i(aD_iv(x))=D_i(f_i\eta+auD_i\eta)-f_iD_i\eta+aD_i uD_i\eta\label{eq 1.212}
\end{equation}
in $\bR^2$.  We define
 \begin{equation}
\tilde{u}(x)=-\int_{\bR^2}D_{y_i}G (x,y)\tilde{f}_i(y)
+G (x,y)(f_i(y)D_i\eta(y)-aD_iuD_i\eta)\,dy,\label{eq 122.2}
 \end{equation}
 where
 \begin{equation}\label{eq 122.1}
 \tilde{f}_i=f_i\eta+auD_i\eta.
 \end{equation}
Since $G $ is Green's function of $L$, the function $\tilde{u}$ defined above is a solution to  \eqref{eq 1.212}.
When $u$ is restricted to $\{|x|>\frac{l}{2}\}$,  the result follows from Case 1. Therefore,  it remains to estimate $u|_{B_{l/2}}$.
Since $v=u$ in $B_{l/2}$,  it suffices to consider $v$ instead of $u$.
Because $$D_i(a(x)D_i(\tilde{u}-v))=0$$
in $\bR^2$, by Theorem \ref{thm 3.151} with a sufficiently large $R_0$, we know that $\tilde{u}-v$ is piecewise smooth. Hence, it suffices to estimate $\tilde{u}$ instead of $v$.

Note that $\text{supp}(D_i\eta)\subset \{l/2<|x|<l\}$, on which by  Case 1 $u\in C^{n+1,\gamma}$ piecewise.
Combining with the definition of $\tilde{f}_i$ in \eqref{eq 122.1}, we get that $\tilde{f}_i$ are piecewise  $C^{n,\gamma}$.  By \eqref{eq 122.2},  
\begin{align*}
&\tilde{u}(x)=-\int_{\bR^2}D_{y_i}G (x,y)\tilde{f}_i(y)\,dy
-\int_{\bR^2}G (x,y)\big(f_i(y)D_i\eta(y)-aD_i uD_i\eta\big)\,dy\\
&:=u_1(x)+u_2(x).
\end{align*}
Since the estimates of $u_1$ and $u_2$  are quite similar, we only consider $u_1$ as follows:
\begin{align*}
&u_1(x)=-\int_{\bR^2}D_{y_i}G (x,y)\tilde{f}_i(y)\,dy\\
&=-\int_{\fB_1}D_{y_i}G (x,y)\tilde{f}_i(y)\,dy-\int_{\fB_2}D_{y_i}G (x,y)\tilde{f}_i(y)\,dy-\int_{\fB_0}D_{y_i}G (x,y)\tilde{f}_i(y)\,dy\\
&:=-w_1(x)-w_2(x)-w_3(x).
\end{align*}
We focus on the case $x\in \Omega_0\cap B_1$ and the same argument can be applied to the other cases as well.
By the definition of $G (x,y)$, we have
\begin{align}\nonumber
&w_1(x)=(1+\alpha)\sum_{k=0}^\infty(\alpha\beta)^{k}
\Big(\int_{\fB_1}D_{y_i}\log|X_{-2k}(x)-{y}|\tilde{f}_i(y)\,dy\\\nonumber
&\quad\quad\quad -\beta\int_{\fB_1}D_{y_i}\log|X_{2k+1}(x)-\overline{y}|\tilde{f}_i(y)\,dy\Big)\\
&=(1+\alpha)\sum_{k=0}^\infty(\alpha\beta)^{k}
\big[h(X_{-2k}(x))-\beta
\hat{h}(X_{2k+1}(x))\big],\label{eq 2.181}
\end{align}
where
\begin{equation}
h(x)=\int_{\fB_1}D_{y_i}\log|x-y|\tilde{f}_i(y)\,dy,\quad
\hat{h}(x)=\int_{\fB_1}D_{y_i}\log|x-\overline{y}|\tilde{f}_i(y)\,dy.\label{eq 3.161}
\end{equation}

Since $\log|x-y|$ is the fundamental solution of the Laplace equation in $\bR^2$, $h(x)$ satisfies
\begin{equation*}
\Delta h=-D_i(\tilde f_i\chi_{\fB_1}) \quad \text{in}\,\, \bR^2.
\end{equation*}
Since  $\tilde f$ is piecewise $C^{n,\gamma}$ and the interface  $\partial\fB_1$ is smooth by the same method in dealing with Case 1, we  obtain that $h\in C^{n+1,\gamma}(B_1)$ piecewise.

For any $k\neq 0$, by the definition,
\begin{equation*}
X_{k}(x_1,x_2)=\left(\Re \frac \ri{\ri/z+k},\Im \frac \ri{\ri/z+k}\right)
=\frac 1 {k}\left(\Re \frac 1{kz+\ri},1+\Im \frac 1{kz+\ri}\right),
\end{equation*}
where $z=x_1+\ri x_2$.
Notice that for $x\in \fB_0$,
$|kz+\ri|\ge 1$. 
By a straightforward  calculation, it is  easily seen that for any $j>0$ and $x\in B_1\cap\fB_0$,
\begin{equation}
|D^j X_{k}(x)|\le C|k|^{j+1},\label{eq 2.221}
\end{equation}
where $C$ is independent of $k$.  Moreover, since $x\in \Omega_0\cap B_1$, $\Theta(x)\in \{|x|> 1\}\cap \{x_1\in (-\.5,\.5)\}$, which implies that $\Theta(x)-{\bf 2k}\in \{x\,:\,x_1<1/2,|x|>1\}$, and
\begin{equation}\label{eq 2.222}
X_{-{\bf 2k}}(x)=\Theta(\Theta(x)-{\bf 2k}))\in B_1\setminus \overline{\fB_1}.
\end{equation}
Therefore, combining \eqref{eq 2.221} and \eqref{eq 2.222}, with the chain rule, we have for any $k\ge 0$
\begin{equation*}
\|h(X_{-2k}(x))\|_{n+1,\gamma;B_1\cap\Omega_0}
\le C(k+1)^{n+3}\Big(\|h\|_{n+1,\gamma; B_1\cap \fB_2}+\|h\|_{n+1,\gamma;B_1\cap\fB_0}\Big).
\end{equation*}

It remains to estimate $\hat{h}(X_{2k-1})$. Notice that
\begin{equation*}
\Big|\int_{\fB_1}D_{y_i}\log|x-\overline{y}|\tilde{f}_i(y)\,dy\Big|
=\Big|\int_{\fB_2}D_{y_i}\log|x-y|\tilde{f}_i(\overline{y})\,dy\Big|,
\end{equation*}
implying that
$$\Delta\hat{h}(x)=D_i\big(\tilde{f}_i(\overline{x})\chi_{\fB_2}\big).$$
Note that for any $k\ge 1$ and $x\in B_1\cap\fB_0$,
\begin{equation*}
X_{2k-1}(x)=\Theta(\Theta(x)+{\bf 2k-1})\in B_1\cap \fB_1.
\end{equation*}
Furthermore, the regularity of $\tilde{f_i}(\overline{y})$ is the same as $\tilde{f_i}(y)$, which indicates that $\hat{h}(X_{2k-1})$ can be estimated in a similar way as $h(X_{-2k})$.
Hence, combining the estimate of $h$ and $\hat{h}$, we have
\begin{align}\nonumber
&\|w_1\|_{n+1,\gamma;B_1\cap\fB_0}\nonumber\\
&\le C(1+\alpha)\sum_{k=0}^\infty(\alpha\beta)^k\Big(\|h(X_{-2k})\|_{n+1,\gamma;B_1
\cap\fB_0}+\beta\|\hat{h}(X_{2k+1})\|_{n+1,\gamma;B_1
\cap\fB_0}\Big)\nonumber\\
&\le C\sum_{k=0}^\infty(\alpha\beta)^{k}(k+1)^{n+3}\Big(\|h\|_{n+1,\gamma;B_1\cap \fB_2}+\|h\|_{n+1,\gamma;B_1\cap\fB_0}+\|\hat{h}\|_{n+1,\gamma; B_1\cap \fB_1}\Big)\nonumber\\
\nonumber
&\le C\sum_{j=1}^2\Big(\|f_j\eta\|_{n,\gamma;\fB_1}+\|uD_j\eta\|_{n,\gamma;\fB_1}\Big)
+C\Big(\|h\|_{L_\infty(B_2)}+\|\hat{h}\|_{L_\infty(B_2)}\Big)\\\label{eq 4.51}
&\le C\Big(\sum_{j=1}^2\|f_j\|_{n,\gamma;\fB_1}+\|u\|_{n,\gamma;\fB_1\cup(B_l\setminus B_{l/2})}\Big),
\end{align}
where the last term on the right-hand side is estimated in Case 1.
Here in the last inequality above, we use the fact that
\begin{align*}
&\|h\|_{L_\infty(B_2)}+\|\hat{h}\|_{L_\infty(B_2)}\\
&\le
 C\sum_{j=1}^2\|\tilde{f}_j\|_{L_\infty(\fB_1)}\le C\Big(\sum_{j=1}^2\|f_j\|_{L_\infty(\fB_1)}+\|u\|_{L_\infty(B_l\setminus B_{l/2})}\Big),
\end{align*}
which can be deduced directly from \eqref{eq 3.161}.

By symmetry, it is easy to see that the argument applied to $w_1$ can be implemented to $w_2$ as well. We omit the details. Now let us estimate $w_3$ with a little modification.  Similar to the expression in \eqref{eq 2.181}, we have
\begin{align*}
w_3(x)&=g(x)+\sum_{k=1}^\infty(\alpha\beta)^{k}
\Big(g(X_{2k}(x))+g(X_{-2k}(x))\\
&\quad +\beta\hat{g}(X_{2k+1}(x))
+\alpha\hat{g}(X_{-(2k+1)}(x))\Big),
\end{align*}
where
\begin{equation*}
g(x)=\int_{\fB_0}D_{y_i}\log|x-y|\tilde{f}_i(y)\, dy,\quad\hat{g}(x)=-\int_{\fB_0}D_{y_i}\log|x-y|\tilde{f}_i(\overline{y})\,dy.
\end{equation*}
Since $\tilde{f}_i(\overline{y})$ has the same regularity as $\tilde{f}_i$, it is sufficient to consider $g$, which satisfies
$$
\Delta g=-D_i(\tilde f_i\chi_{\fB_0}).
$$
Here, we cannot directly apply the result in Case 1 because of the singularity of the domain $\fB_0$.
Nonetheless, by Lemma \ref{lemma 3.231}, there exists $F_i\in C^{n,\gamma}(\bR^2)$ which is the extension of $\tilde f_i \chi_{\fB_0}$ and satisfies
$$
\|F_i\|_{n,\gamma;\bR^2}\le C\big(\|f_i\|_{n,\gamma; B_l\cap\fB_0}+\|u\|_{n,\gamma;(B_l\setminus B_{l/2})\cap\fB_0}\big).
$$
Therefore, define
\begin{equation*}
\tilde{g}:=\int_{\bR^2}D_{y_i}\log|x-y|F_i(y)\,dy,\quad \mathfrak{g}_1:=\int_{\fB_1}D_{y_i}\log|x-y|F_i(y)\,dy,
\end{equation*}
and
\begin{equation*}
\mathfrak{g}_2:=\int_{\fB_2}D_{y_i}\log|x-y|F_i(y)\,dy,
\end{equation*}
which satisfy
\begin{equation*}
\Delta\tilde{g}=-D_iF_i,\quad \Delta\mathfrak{g}_1=-D_i(F_i\chi_{\fB_1}),\quad
\text{and}\quad \Delta\mathfrak{g}_2=-D_i(F_i\chi_{\fB_2}).
\end{equation*}
 From the classical Schauder estimate, we have $\tilde{g}\in C^{n+1,\gamma}(B_1)$.
By the estimate of $w_1$ above, we have $\mathfrak{g}_1$, $\mathfrak{g}_2$ are all piecewise $C^{n+1,\gamma}$, which implies $g=\tilde{g}-\mathfrak{g}_1-\mathfrak{g}_2$ is piecewise $C^{n+1,\gamma}$ as well. Then we can follow the same argument in the estimate of $w_1$ (cf. \eqref{eq 4.51}) and obtain a similar estimate for $w_3$
\begin{equation*}
\|w_3\|_{n+1,\gamma;B_1\cap \fB_0}\le C\big(\sum_{j=1}^2\|f_j\|_{n,\gamma;B_l\cap\fB_0}+\|u\|_{n,\gamma;(B_l\setminus B_{l/2})\cap\fB_0}\big).
\end{equation*}
Hence, we show that $\tilde{u}$ is piecewise $C^{n+1,\gamma}$ and the proof is completed.
 \end{proof}

\subsection{Two balls with different radii}
                            \label{sec4.3}
Next, we consider the general case that the two balls have different radii. Specifically,
\begin{equation*}
a(x)=a_0\chi_{B_{r_1}(0,r_1)}+b_0\chi_{B_{r_2}(0,-r_2)}
+\chi_{\mathcal{D}\setminus  (B_{r_1}(0,r_1)\cup B_{r_2}(0,-r_2))}.
\end{equation*}
Denote $\partial B_{r_1}(0,r_1)=C_1$ and $\partial B_{r_2}(0,-r_2)=C_2$.
By scaling and reflection, without loss of generality, we may assume that $r_2>r_1>1/2$.
Now, we consider a conformal map $\tilde{T}:\bC\rightarrow \bC, \tilde{T}(z)= \frac{1}{z-z_0}$, where $z_0\in \bC$  and $z_0=z_1+z_2\ri$.
It is well known that for $z_0\notin C_1\cup C_2$, $\tilde{T}$ maps $C_1$ and $C_2$ to two circles. We shall find a suitable $z_0$ such that the two circles have the same radius. Indeed, by a simple computation it is easy to see that $\tilde{T}(C_1)$ and $\tilde{T}(C_2)$ are circles with radii $|r_1/(z_1^2+z_2^2-2r_1z_2)|$ and $|r_2/(z_1^2+z_2^2+2r_2z_2)|$, respectively. 
Then, we only need to find $(z_1,z_2)$ such that
$$
\big|r_1/(z_1^2+z_2^2-2r_1z_2)|=|r_2/(z_1^2+z_2^2+2r_2z_2)\big|.
$$
It is obvious that
\begin{align}\nonumber
r_1(z_1^2+z_2^2+2r_2z_2)=r_2(z_1^2+z_2^2-2r_1z_2)\\
 \Leftrightarrow (r_2-r_1)(z_1^2+z_2^2)=4z_1z_2r_1r_2\label{eq 1.211}.
\end{align}
Note that \eqref{eq 1.211} can be written as
\begin{equation}
\frac{z_1}{z_2}+\frac{z_2}{z_1}=\frac{4r_1r_2}{r_2-r_1}>2,\label{eq 3.241}
\end{equation}
which implies the existence of $(z_1,z_2)$. Moreover, if $(z_1,z_2)$ is a solution of \eqref{eq 3.241}, for any $s>0$ $(sz_1,sz_2)$ solves \eqref{eq 3.241} as well. Hence, we can pick $z_0$ outside $\mathcal{D}$, which indicates that $\tilde{T}$ is smooth in $\mathcal{D}$. After a translation, rotation, and scaling of the coordinates, we can assume that $\tilde{T}$ maps $B_{r_1}(0,r_1)$ and $B_{r_2}(0,-r_2)$ to $\fB_1$ and $\fB_2$, respectively.
Therefore, we obtain the desired conformal map $\tilde{T}$, which is a diffeomorphism between $\mathcal{D}$ and $\tilde{T}(\mathcal{D})$.

Now we are ready to prove Theorem \ref{thm 2.161}.
\begin{proof}[Proof of Theorem \ref{thm 2.161}]
In order to consider the equation
\begin{equation*}
D_i(a(x)D_iu(x))=D_if_{i}(x),
\end{equation*}
we define $v(x)=u(\tilde{T}^{-1}x)$ in $\tilde{T}(\mathcal{D})$.  By a straightforward calculation, we have
\begin{equation*}
D_i(a(\tilde{T}^{-1}(x))D_iv)=D_k(\eta_{ki}f_i(\tilde{T}^{-1}x))
\end{equation*}
in $\tilde{T}(\mathcal{D})$, where $\eta_{ki}$ are some smooth functions in $\tilde{T}(\mathcal{D})$.
The operator on the left-hand side of the equation above is the same as the operator in Theorem \ref{thm 2.191}. Therefore, applying Theorem \ref{thm 2.191}, we get $v\in C^{n+1,\gamma}$ piecewise in $\tilde{T}(\mathcal{D})\cap B_{l}$, where $B_l$ is a small ball around the origin. This, combined with the smoothness of $\tilde{T}$ and the chain rule,  implies that $u\in C^{n+1,\gamma}$ piecewise in a small ball $B_{\hat{l}}\subset \tilde{T}^{-1}(B_l)$ around the origin. For the region outside $B_{\hat{l}}$, it follows from Case 1 in the proof of Theorem \ref{thm 2.191}. Hence,  the theorem is proved.
\end{proof}

By a standard freezing coefficients argument, finally we prove Corollary \ref{cor 3.71}.
\begin{proof}[Proof of Corollary \ref{cor 3.71}]
First, we consider a ball $B_{l}(x_0)$ such that $B_l(x_0)\subset \mathcal{D}_\varepsilon\cap B_{r_1}(0,r_1)$. By the classical Schauder estimate, we have
\begin{equation}\label{eq 3.72}
\|u\|_{n+1,\gamma;B_{l/2}(x_0)}\le C\Big(\|u\|_{L_\infty(B_l(x_0))}+\sum_{j=1}^2\|f_j\|_{n,\gamma;B_l(x_0)}\Big).
\end{equation}
Similarly, we can show \eqref{eq 3.72} holds if $B_l(x_0)\subset \mathcal{D}_\varepsilon\cap B_{r_2}(0,-r_2)$ or $B_l\subset\mathcal{D}_\varepsilon\setminus (B_{r_1}(0,r_1)\cup B_{r_2}(0,-r_2))$.

Second, we consider a ball $B_l(x_0)$ such that $x_0\in \partial (B_{r_1}(0,r_1)\cup B_{r_2}(0,-r_2))$ but $0\notin B_{l}(x_0)$. By flattening the boundary, this is essentially the same as Case 1 in the proof of Theorem \ref{thm 2.191}.  This implies that $u$ is piecewise $C^{n+1,\gamma}$ in this case.

It remains to consider $B_l(x_0)=B_l(0,0)\subset \mathcal{D}_\varepsilon$. In this case, we define $\tilde{a}$ as follows:
\begin{align*}
&\tilde{a}(x)=\lim_{\substack {x\rightarrow (0,0) \\x\in B_{r_1}(0,r_1)}}a(x),\quad x\in B_{r_1}(0,r_1),\\
&\tilde{a}(x)=\lim_{\substack {x\rightarrow (0,0) \\x\in B_{r_2}(0,-r_2)}}a(x),\quad x\in B_{r_2}(0,-r_2),\\
&\tilde{a}(x)=\lim_{\substack {x\rightarrow (0,0) \\
x\in \bR^2\setminus (B_{r_1}(0,r_1)\cup B_{r_2}(0,-r_2))}}a(x),\quad \text{otherwise.}
\end{align*}
Then the equation can be written as
\begin{equation}
D_i(\tilde{a}D_iu)=D_i(f_i+(\tilde{a}-a)D_iu).\label{eq 3.81}
\end{equation}
Denote $\hat{\Omega}_1=B_{l}\cap B_{r_1}(0,r_1)$, $\hat{\Omega}_2=B_{l}\cap B_{r_2}(0,-r_2)$, and $\hat{\Omega}_0=B_{l}\setminus(B_{r_1}(0,r_1)\cup B_{r_2}(0,r_2))$.
Let $\Omega_k=\hat{\Omega}_k\cap B_{l/2}$ for $0\le k\le 2$.
By the Leibniz rule and the interpolation inequality, we get
\begin{equation*}
[(a-\tilde{a})Du]_{n,\gamma;\hat{\Omega}_k}\le C(\varepsilon)\|u\|_{L_\infty(\hat{\Omega}_k)}
+(l^\gamma+\varepsilon)[u]_{n+1,\gamma;\hat{\Omega}_k},
\end{equation*}
where $\varepsilon>0$ and $C(\varepsilon)$ depends on $\varepsilon, n$, and $\|a\|_{n,\gamma;\Omega_k}$.

 Applying Theorem \ref{thm 3.133} to  \eqref{eq 3.81},  we obtain
\begin{align}\nonumber
&\sum_{k=0}^2\|u\|_{n+1,\gamma;\Omega_k}\\
&\le C(\varepsilon)\Big(\sum_{k=0}^2\sum_{j=1}^2\|f_j\|_{n,\gamma;\hat{\Omega}_k}+
\sum_{k=0}^2\|u\|_{L_\infty(\hat{\Omega}_k)}\Big)+C_1(l^\gamma+\varepsilon)\sum_{k=0}^2[u]_{n+1,\gamma;\hat{\Omega}_k},\label{eq 3.285}
\end{align}
where $C_1$ is independent of $\varepsilon$.
Since
\begin{equation*}
[u]_{n+1,\gamma;\hat{\Omega}_k}\le C[u]_{n+1,\gamma;\Omega_k}+C[u]_{n+1,\gamma;\hat{\Omega}_k\setminus\Omega_k},
\end{equation*}
where $C$ is independent of $u$ and $l$, by taking $l$ and $\varepsilon$ sufficiently small in \eqref{eq 3.285}, we obtain
\begin{equation*}
\sum_{k=0}^2\|u\|_{n+1,\gamma;\Omega_k}\le C\Big(\sum_{k=0}^2\sum_{j=1}^2\|f_j\|_{n,\gamma;\hat{\Omega}_k}
+\sum_{k=0}^2\|u\|_{L_\infty(\hat{\Omega}_k)}\Big)+C\sum_{k=0}^2[u]_{n+1,\gamma;\hat{\Omega}_k\setminus\Omega_k},
\end{equation*}
where $[u]_{n+1,\gamma;\hat{\Omega}_k\setminus\Omega_k}$ is  estimated in the previous case.
Therefore, thanks to the argument of partition of the unity,   the corollary is proved.
\end{proof}

\bibliographystyle{plain}
\bibliography{elliptic_pde}

\end{document}